\documentclass[12pt]{article}

\usepackage[toc,page]{appendix}

\usepackage{etoolbox}
\appto\appendix{\addtocontents{toc}{\protect\setcounter{tocdepth}{1}}}

\usepackage{enumerate}
\usepackage[margin=2.4cm]{geometry}
\usepackage{t1enc}
\usepackage[utf8]{inputenc}
\usepackage{amsthm,amsmath,amssymb}
\usepackage{graphicx}
\usepackage{hyperref}
\usepackage{bm}
\usepackage{natbib}
\usepackage{color}

\usepackage{amsfonts}
\usepackage{graphicx}
\usepackage{bm}
\usepackage{amsmath, amsthm, amssymb}
\usepackage{graphicx}
\usepackage{hyperref}
 \usepackage{relsize}
 \usepackage{verbatim}
\usepackage{nicematrix}
\DeclareMathOperator{\Var}{Var}

\usepackage{bbm}

\definecolor{ao(english)}{rgb}{0.0, 0.5, 0.0}

\theoremstyle{plain}

\newtheorem{theorem}{Theorem}
\newtheorem*{theorem-non}{Theorem}
\newtheorem{lemma}[theorem]{Lemma}

\newtheorem{proposition}[theorem]{Proposition}

\newtheorem{conjecture}[theorem]{Conjecture}

\theoremstyle{definition}
\newtheorem{remark}[theorem]{Remark}

\newtheorem{problem}[theorem]{Problem}
\newtheorem{question}[theorem]{Question}
\newtheorem{corollary}[theorem]{Corollary}

\newcommand{\ux}{{\mathfrak{a}}}
\newcommand{\uy}{{\mathfrak{b}}}

\DeclareMathOperator{\Tr}{Tr}

\newcommand{\ev}{\mathbb{E}}

\title{
Eigenvectors of the square grid plus GUE}
\author{Andr\'as M\'esz\'aros \and B\'alint Vir\'ag}

\begin{document}

\maketitle

\begin{abstract}
Eigenvectors of the GUE-perturbed discrete torus with uniform boundary conditions retain some product structure for small perturbations but converge to  discrete Gaussian waves for large perturbations. We determine where this phase transition happens. 
\end{abstract} 
\begin{center}
    \includegraphics[width=6in]{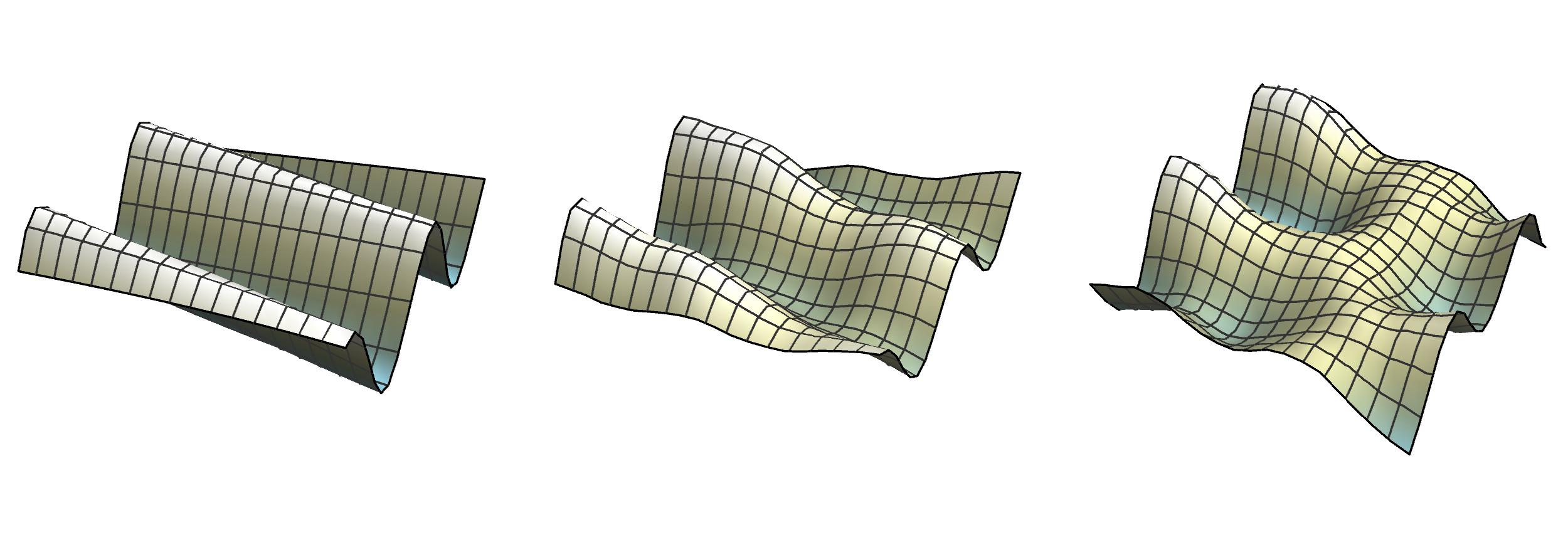}
\end{center}

\tableofcontents
\section{Introduction}


High eigenfunctions of the Laplacian in irregular planar domains look locally like Gaussian  waves. This phenomenon, is a manifestation of quantum ergodicity: interactions of eigenfunctions result in  averaging phenomena akin to the ergodic theorem. 

Such interactions are not expected to take place in regular domains. For example, in a discrete $n\times (n+1)$ box  most eigenfunctions are products $f(x)g(y)$ of eigenfunctions of the two coordinate paths,  and the box eigenvalue is the sum of the corresponding two path eigenvalues. 

Perturbations of the Laplacian can cause Gaussian waves to appear even in the regular setting. An example is the Anderson model on the discrete torus, where even for low levels of noise Gaussian waves are expected to appear. 

The goal of this paper is to understand this phenomenon for the discrete grid perturbed by a GUE random matrix.  

\begin{question} How much noise do we need to add to a discrete torus so that the eigenfunctions become Gaussian waves in the limit?
\end{question}

We answer to this question for the discrete torus with uniform boundary conditions perturbed by a GUE matrix. Let us introduce the main concepts.

\medskip

{\bf The Gaussian wave.} The adjacency matrix $A_{\mathbb Z^2}$ of the planar square grid $\mathbb Z^2$ has pure continuous spectrum with support $[-4,4]$. At any  $E\in(-4,4)$  the spectrum has infinite multiplicity. While there are no $\ell^2$ eigenfunctions, any function of the form $\exp(i(\alpha x+\beta y))$ where $2\cos(\alpha)+2\cos(\beta)=E$  satisfies the eigenvalue equation.  The complex Gaussian wave $Z_E$ is the canonical random linear  combination  of such wave functions. It is defined as the complex Gaussian process on $\mathbb Z^2$ with covariance  matrix given by the density of the matrix-valued spectral measure of $A_{\mathbb Z^2}$ at $E$. See Section \ref{subsecwave} for an explicit formula. 

{\bf The Gaussian unitary ensemble.} The GUE is a random Hermitian matrix with independent entries. It can be realized as \begin{equation}\label{GUEdef}W_n=M+M^*\end{equation} where $M$ is a matrix of independent centered complex Gaussian entries of variance $1/(2n)$.

{\bf The discrete torus with uniform boundary conditions.} The graph of the discrete $n$-cycle has the adjacency matrix $A_{n-cycle}$ with $(A_{n-cycle})_{i,j}=1$ when $|i-j|=1$ mod $n$ and $0$ otherwise. Introduce a version $A_{n,c}$ with boundary condition $c$ by replacing $(\overline{A_n})_{1,n}$ and $ ( A_n)_{n,1}$ by the value  $e^{2\pi ic}$. The discrete torus with  boundary conditions $c,d$ is the Cartesian product of the corresponding weighted graphs. Its adjacency matrix is thus given by the tensor product expression  $A_{n,c,d}=A_{n,c}\otimes I  + I\otimes A_{n,d}$. When  $c,d$ are independent uniform $[0,1]$ random variables, we will denote this (slightly) random adjacency matrix by $A_n$. 

\begin{theorem}\label{t:main}
Let $\gamma>0$, $\delta\in (0,1)$, $E\in (-4,4)\setminus \{0\}$.  Consider eigenvalues of 
$$
A_n+n^{-\gamma}W_{n^2},
$$
(the adjacency matrix of the discrete torus with uniform boundary conditions plus an independent scaled GUE) in the interval  $(E\pm n^{-\delta})$. Let $u_1,u_2,\dots,u_m$ be the corresponding eigenfunctions
with independent uniform random phases and $\ell^2$-norm $n$. Let $o_n\in \{1,\ldots,n\}^2$ be a sequence so that for both coordinates $j$ we have $(o_n)_j,n-(o_n)_j\to \infty$. Let $n\to\infty$ and consider the random measures 
\[\frac{1}{m} \sum_{i=1}^m \delta_{u_i(\cdot+o_n)}.\]
\begin{itemize}
    \item If $\gamma<1$, then for all $\delta$ sufficiently close to $0$, the measures converge in probability  to the law of the complex Gaussian wave $Z_{E}$.
    \item If $\gamma>1$, then for all $\delta<1$, then the law of $Z_E$ is not a limit point of these measures. 
\end{itemize}
\end{theorem}

\begin{figure}
        \centering
    \includegraphics[width=5in]{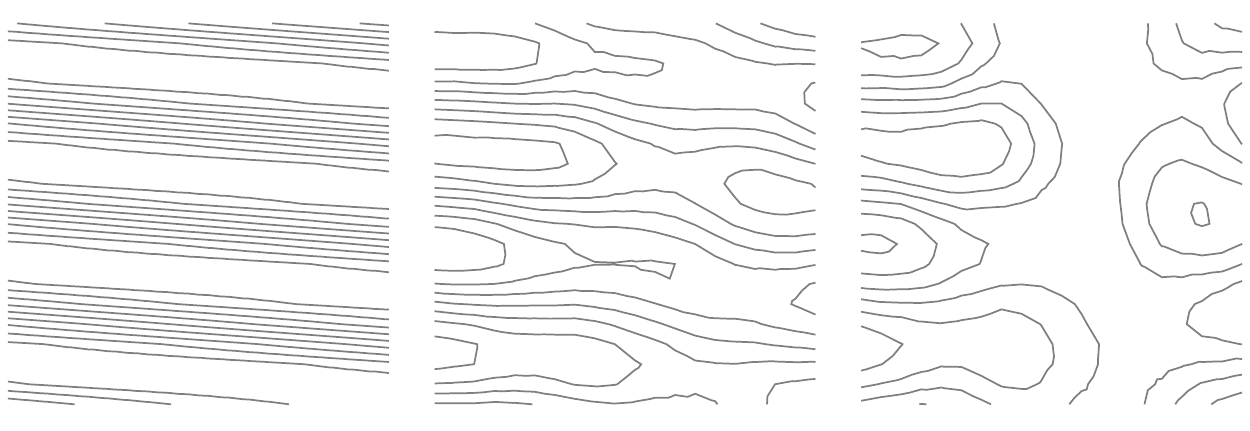}
        \caption{Level sets of the real parts of eigenfunctions on page 1 as the noise increases}
        \label{f:eigenfunctions}
\end{figure}

\begin{remark}In terms of local observables, convergence in probability  to the law of $Z_{E}$ means that for any bounded continuous functions $f:\mathbb{C}^{\mathbb{Z}^2}\to\mathbb{R}$, 
\[\frac{1}{m} \sum_{i=1}^m f(u_i(\cdot+o_n))\]
converge to $\mathbb{E} f(Z_{E})$ in probability.
\end{remark}

We do not consider $E=0$ since the density of the spectral measure of $\mathbb Z^2$ blows up there, see also Problem~\ref{atE=0}.

For random regular graphs, the Gaussian limiting behaviour of the eigenvectors was established by \cite{backhausz2019almost}. Note that for random regular graphs the randomness comes from the random choice of the underlying graph, but in our results the underlying graph is fixed and the randomness is introduced by the random boundary and by adding the noise. See also \cite{bauerschmidt2017local,bauerschmidt2017bulk,bauerschmidt2019local,bauerschmidt2020edge,huang2021spectrum,huang2023edge} for more results on the spectrum of random regular graphs. \cite{anantharaman2015quantum} proved a version of the Quantum Ergodicity theorem for deterministic $d$-regular expander graphs with large essential girth. \cite{bourgain2014toral}  shows the emergence of Gaussian waves  for high-dimensional eigenspaces in the unperturbed continuum torus. For a recent continuous version of Gaussian random waves appearing as high energy eigenfunctions see \cite{abert2018eigenfunctions,ingremeau2021local}.

Our  tool for proving convergence to Gaussian waves is Theorem~\ref{lemmaquenched1A} stated in Section~\ref{sectech}. A central ingredient of the proof is the theorem of \cite{benigni2020eigenvectors} on eigenvector limits, see Theorem~\ref{benignithm}. The novelty of Theorem~\ref{lemmaquenched1A} compared to  Benigni's result is that it works in the locally convergent matrix setting and the regularity of the empirical measure of the eigenvalues is only required at most points of a given window instead of at all the points of the window. Moreover,  Theorem~\ref{lemmaquenched1A} is easier to apply as it does not use the notion of free convolution.

Benigni's result relies on the theory of the eigenvector moment flow developed by  \cite{bourgade2017eigenvector} and \cite{bourgade2018random}, and on a few results on local laws, see \cite{landon2017convergence} and \cite{lee2016bulk}. Similar results were proved by \cite{marcinek2020high}.
For further results on Quantum Unique Ergodicity for Wigner matrices, see \cite{benigni2022fluctuations,benigni2022optimal,cipolloni2021eigenstate,cipolloni2022normal}.

\label{generallocalweak}

To apply Theorem~\ref{lemmaquenched1A} for adjacency matrix $A$ of the discrete torus with uniform boundary conditions, we need a good control over the spectral measure of $A_{{\rm{c}},{\rm{d}}}=A_{n,{\rm{c}},{\rm{d}}}$ down to the optimal scale $n^{-2+\varepsilon}$ for most choices of the boundary conditions. The following variance estimate allows us to compare the spectral measure of $A_{{\rm{c}},{\rm{d}}}$ to the spectral measure of the adjacency matrix of the square grid.

Given $(o_1,o_2),(o_1+\ux,o_2+\uy)\in \{1,2,\dots,n\}^2$, let $\mu_{{\rm{c}},{\rm{d}},\ux,\uy}=\mu_{n,{\rm{c}},{\rm{d}},\ux,\uy}$ be the unique complex-valued measure such that
\[f(A_{{\rm{c}},{\rm{d}}})\left((o_1+\ux,o_2+\uy),(o_1,o_2)\right)=\int f(x) \mu_{{\rm{c}},{\rm{d}},\ux,\uy}(dx)\]
for all continuous functions $f$. It turns out that this measure   depends on $\ux,\uy$ but not on~$(o_1,o_2)$.

\begin{theorem}\label{torusthm1}
For every $\frac{1}{4}>\epsilon>0$ there is $c$ so that the following holds for all $n$, $n^2\ge \eta\ge n^{-2}$ and $|\ux|+|\uy|\le n^{1/2-2\varepsilon}$. Let $C$ and $D$ be independent uniform random elements of $[0,1]$, then
\begin{equation}
\label{e:torus}
\left|\frac{\pi}{2n^ 2 \eta} -\int_{-\infty}^\infty \Var\left(\int \frac\eta{(x-\lambda)^2+\eta^2} \mu_{C,D,\ux,\uy}(dx)\right)d\lambda\right| \le  c\frac{1+|\ux|+|\uy|}{\eta^\epsilon \sqrt{n}}.
\end{equation}
\end{theorem}

We will see later that $n^2 \mu_{0,0,0,0}([E-\eta,E+\eta)$ is given by the number of lattice points in the region
\[\{(x,y)\in [-n/2,n/2)^2\,:\,E-\eta\le 2\cos(2\pi x/n)+2\cos(2\pi y/n)\le E+\eta\}.\]
Thus, we arrive to a problem very similar to the \textbf{Gauss circle problem}, where we are interested in the number of lattice points inside a circle of radius $r$. It is conjectured that the number of lattice points is equal to $\pi r^2+F(r)$, where $|F(r)|=O(r^{1/2+\varepsilon})$. The current best bound is $|F(r)|\le O(r^{131/208})$ proved by \cite{huxley2002integer}. A closely related problem is to count the number of lattice points in a thin annulus. These problems become more manageable if we introduce some randomness to them. For example, we can choose the radius randomly, see the results of \cite{bleher1993distribution} for the circle and \cite{bleher1995variance} for the annulus. Another option is to shift the lattice with a uniform random vector in $[0,1]^2$, see the results of \cite{cheng1994number}. It turns out that for the discrete torus, our choice of random boundary also corresponds to shifting the lattice. Thus, Theorem~\ref{torusthm1} fits into the framework of these problems.


To prove the non-convergence part of Theorem~\ref{t:main}, we use Theorem~\ref{lemmaconc01}  on the concentration of perturbed eigenvectors, which is stated in Section~\ref{sectech}.

The proof of this theorem relies on the GUE resolvent flow, which was used by \cite{von2018phase} to obtain similar concentration results, although they measured the concentration in a different way. Their proof uses the spectral averaging principle and other related estimates (\cite{minami1996local,kotani1984lyaponov,combes2009generalized}). Thus, they need to assume that the initial eigenvalues are random satisfying certain conditions. However, it turns out that this random choice is not essential, our Theorem~\ref{lemmaconc01} is for deterministic initial conditions. For $n^{-1}\ll t\ll 1$, similar results were proved by \cite{von2019non}.

\subsection{Open problems}

We believe that the random matrix nature of the perturbations in Theorem \ref{t:main} is not relevant. 

\begin{problem}
Is the critical value of $\gamma$ still $1$ if we replace $W$ by a diagonal matrix with independent standard Gaussian entries?
\end{problem}

\begin{problem}\label{prob:standard}
 Does Theorem \ref{t:main} hold on the standard torus, without the uniform boundary conditions?
\end{problem}

\begin{problem}\label{atE=0}
What happens at $E=0$? Although $\lim_{x\to 0} \varrho_{0,0}=\infty$, it appears that
\[\lim_{x\to 0} \frac{\varrho_{\ux,\uy}}{\varrho_{0,0}}\]
exists for all $\ux,\uy\in \mathbb{Z}$. Using this limit as the right hand side of \eqref{eqnewMdef}, we obtain a good candidate for the covariance structure of the limiting eigenvalue process.

Also, what happens at $E=\pm 4$? At $E=4$, we expect to see a constant vector with a uniform phase. At $E=-4$, we need to add signs with a checkerboard pattern.
\end{problem}

The uniform boundary conditions help us prove regularity of the unperturbed eigenvalues. The main obstacle to Problem \ref{prob:standard} is closely related to the following conjecture. Let $\xi_i$ be independent uniformly chosen $n$th roots of unity. 

\begin{conjecture}
$$
\mathbb{P}(|\Re (\xi_1+\xi_2+\xi_3+\xi_4)|<1/n^2)=1/n^{2+o(1)},\qquad n\to \infty.
$$\end{conjecture}
See the discussion in Section 3.

\begin{problem}  Let $D$ be a Jordan domain in the plane, and let $A_n$ be the adjacency matrix of $(\tfrac{1}{n}\mathbb Z^2)\cap D$. Does the  convergence part of Theorem \ref{t:main} hold? Under what conditions is $\gamma=1$ critical?
\end{problem}

For sufficiently irregular domains, Gaussian waves should appear even without perturbation. 

\begin{problem}
What is the critical value $\gamma_c$ for the 3-dimensional torus? How about general dimensions?
\end{problem}
It seems that our methods can be generalized to higher dimensions, giving $\gamma_c=d/2$, but we did not pursue this direction for the sake of brevity. 

We distinguish eigenvectors in the product phase by their local Fourier transform: it has some dominant coefficients with positive probability.  But we believe that more is true. 
\begin{problem}
In the product phase, $\gamma>1$, show that most eigenvectors converge locally to products. 
\end{problem}

\section{Statements of general perturbations  results}\label{sectech}

The two parts of our main theorem rely on the following more general results about perturbations of deterministic matrices. We expect that these results can useful in more general settings. 

\begin{theorem}[Eigenvector limits of GUE-perturbed locally convergent matrices]\label{lemmaquenched1A}
Let $c>0$, let $A=A_n$ be an $n\times n$ Hermitian matrices with $\|A\|\le c$,  $E=E_n\in \mathbb R$,  $0<\varepsilon<\frac{1}3$, $t=t_n\in [n^{\varepsilon-1},n^{-2\varepsilon}]$, and $\ell=\ell_n>\max(\sqrt{tn^\varepsilon },t n^{2\varepsilon})$. 

Let $\mathbb{P}_{\lambda}$ denote uniform random choice of $\lambda\in [E-\ell,E+\ell]$. 
Let $n\to\infty$, and assume that for 
$m(z)=\Tr((A-z)^{-1})/n$, the Stieltjes transform of the empirical law of eigenvalues,
\begin{align}\label{condregA}n^{2\varepsilon}\log (n)\sup_{\eta\in [t n^{-\varepsilon},20]}{\mathbb{P}_{\lambda}\left( \Im m(\lambda+\eta i)\notin (1/c,c)\right)}\to 0.
\end{align}
Assume that there are $\eta_{\ell}=\eta_{\ell,n}$ and $\eta_u=\eta_{u,n}$ in the interval $(t/c,ct)$ so that $\eta_\ell<\eta_u$, ${\eta_{u,n}}/{\eta_{\ell,n}}\to 1$, the functions $\Im m(\lambda+i\eta_{u})$ are $\mathbb{P}_\lambda$-uniformly integrable and 
\begin{align}\label{cond2A}
n^{2\varepsilon}\mathbb{P}_{\lambda}\left(\Im m(\lambda+i\eta_{\ell})\le\frac{\eta_{\ell}}t+n^{\varepsilon}\right), \quad n^{2\varepsilon}\mathbb{P}_{\lambda}\left(\Im m(\lambda+i\eta_{u})\ge\frac{\eta_{u}}t-n^\varepsilon\right)&\to 0.
\end{align}
Assume that for an $\mathbb N\times \mathbb N$ matrix $M$ 
with  $\eta=\eta_{u}$ and $\eta=\eta_{\ell}$,  
\begin{align}\label{limitMA}n^{\varepsilon}\mathbb{P}_{\lambda}\left( \left|\left(\frac{t}{(A-\lambda)^2+\eta^2}\right)(x,y)-M(x,y)\right|>\delta\right)\to 0 \qquad \mbox{for all }x,y\in \mathbb N\text{ and }\delta>0.
\end{align}

Let $W=W_n$ be an $n\times n$ GUE matrix with  entries of variance $n^{-1}$.
Let 
\[\lambda_{1}\le \dots\le\lambda_{n}\quad \mbox{
be the eigenvalues of }\quad
A+\sqrt{t} W,\]
and let $u_j$ be the corresponding eigenvectors with $\ell^2$-norm $\sqrt{n}$, and  phases  chosen independently uniformly at random. Appending with zeros makes $u_j$ a vector in $\mathbb{C}^\mathbb{N}$.

Then the  empirical measure of eigenvectors 
\[\frac{1}{|\{j\,:\,\lambda_{j}\in [{E}\pm \ell/2]\}|} \sum_{j:\lambda_j\in [ E\pm \ell/2]} \delta_{u_j}\]
converges in probability to the law of the complex Gaussian process on $\mathbb N$ with covariance $M$.
\end{theorem}

See also Theorem~\ref{lemmaquenched1} for a simpler deterministic version of Theorem~\ref{lemmaquenched1A}, where it is assumed that the regularity conditions hold everywhere not just with high probability.

Next we state our concentration result for small perturbation. Note that the conlusion of the next theorem is only meaningful for $t\ll n^{-1}$.

\begin{theorem}\label{lemmaconc01}
Let $A$ be an $n\times n$ Hermitian matrix and let $W$ be an $n\times n$ GUE matrix with  entries of variance $1/n$.
Let 
\[\lambda_{1,t}\le \dots\le\lambda_{n,t}\quad \mbox{
be the eigenvalues of }\quad
A+\sqrt{t} W,\]
and let $u_{j,t}$ be a corresponding orthonormal  basis of eigenvectors. 
Write $u_{k,t}$ in the basis $(u_{j,0})$ to obtain a vector $v_{k,t}$. Let $m(z)=\frac{1}{n} \Tr (A-zI)^{-1}$.    

Let $K_2=\{k:\lambda_{k,t}\in \left[E_\ell,E_u\right]\}$ for some $E_\ell<E_u$. Let $K$ be some set of indices $j$ for which  
\[\lambda_{j,0}\in [E_\ell+1/n,E_u-1/n],\qquad |m(\lambda_{j,0}+i/n)|\le (nt)^{-1/2}.\]
Let $b$ be the number of $\lambda_{k,0}$ in $[E_\ell-3\sqrt{t},E_u+3\sqrt{t}]$. Then for $t\ge e^{-\sqrt{n}}$ we have 
\[\mathbb{E}\,\frac{1}{|K_2|}\sum_{k\in K_2}\sum_{\substack{j\in K\\|\lambda_{k,t}-\lambda_{j,0}|< 1/n}}|v_{k,t}(j)|^2 \ge \frac{|K|}{b}\left(1-c\sqrt{nt}\right),\]
where 
$c$ is an universal constant.
\end{theorem}

\section{Preliminaries}

\subsection{Positive semidefinite operators }
\begin{lemma}Let $H$ be a finite or countably infinite set.\label{posdef}\hfill

\begin{enumerate}[(i)]
    \item Let $A$ be a selfadjoint bounded operator on $\ell^2(H)$ and let $f:\mathbb{R}\to\mathbb{R}_{\ge 0}$ be a non-negative function. Then $f(A)$ is positive semidefinite.
    \item If $M$ is a positive semidefinite operator on $\ell^2(H)$ and $x,y\in H$, then \[M(x,x)\ge 0,\quad M(y,y)\ge 0\text{ and }|M(x,y)|\le \sqrt{M(x,x)\cdot M(y,y)}.\]
    Here $M(x,y)=\langle Me_y,e_x\rangle$, where $e_x$ is the characteristic vector of $x$. In other words, we identified $M$ with its matrix in the standard basis.
\end{enumerate}

\end{lemma}

We need the following consequence of Weyl's inequality.

\begin{lemma}\label{weyl}
Let $A$ and $B$ be two Hermitian  $n\times n$ matrices. Let $\lambda_i$ be the $i$th smallest eigenvalue of $A$, let $\lambda_i'$ be the $i$th smallest eigenvalue of $A+B$.  Then
$|\lambda_i'-\lambda_i|\le \|B\|$.
\end{lemma}

\subsection{The Gaussian unitary ensemble}\label{subsecGUE}
Recall the definition of the Gaussian unitary ensemble from \eqref{GUEdef}.

\begin{lemma}\label{invariant}
If $U$ is a deterministic $n\times n$ unitary matrix,  then $W_n$ and $U^* W_n U$ have the same law. 
\end{lemma}

For the next lemma see \cite{ledoux2010small}.
\begin{lemma}\label{normbecs}
There is  $c>0$ such that for any $n$, we have
$\mathbb{P}(\|W_n\|\ge 3)\le c \exp\left(-{n}/c\right)$.
\end{lemma}

\subsection{The spectral measure of $\mathbb{Z}^2$}\label{subsecspec}

Let $A_{\mathbb{Z}^d}$ be the adjacency matrix of the $d$-dimensional grid. For $x\in \mathbb{Z}^d$, let $e_x\in \ell^2(\mathbb{Z}^d)$ be the characteristic vector of $x$.

The following lemma describes the spectral measure of $A_{\mathbb{Z}}$.

\begin{lemma}
For $\ux\in \mathbb{Z}$, let
\begin{equation}\label{e:arcsine}d_{\ux}(t)=T_{|\ux|}\left(\frac{t}{2}\right)\frac{\mathbbm{1}(|t|\le 2)}{\pi\sqrt{4-t^2}},
\end{equation}
where $T_{|\ux|}$ is the degree $|\ux|$ Chebyshev polynomial of the first kind.
Then for any $x\in \mathbb{Z}$, and any continuous function $f:\mathbb{R}\to\mathbb{R}$, we have 
\[\langle  f(A_{\mathbb{Z}})e_{x},e_{x+\ux} \rangle=\int f(t) d_{\ux}(t)dt.\]
\end{lemma}
\begin{proof}
For $v\in \ell^2(\mathbb{Z})$, let $\hat{v}$ be its Fourier-transform, that is, $\hat{v}(s)=\frac{1}{\sqrt{2\pi}}\sum_{k\in\mathbb{Z}}e^{isk}$ for all $s\in[-\pi,\pi)$. Using the convolution theorem, we see that $\widehat{A_{\mathbb{Z}}v}(s)=\sqrt{2\pi}(\hat{e}_1(s)+\hat{e}_{-1}(s))\hat{v}=2\cos(s)\hat{v}(s)$. Using the fact that $v\mapsto \hat{v}$ is an isometry from $\ell^2(\mathbb{Z})$ to $L^2([-\pi,\pi))$, we have 
\[\langle  f(A_{\mathbb{Z}})e_{x},e_{x+\ux} \rangle=\int_{-\pi}^\pi f(2\cos(s))\hat{e}_{x+\ux}(s)\overline{\hat{e}_x(s)}\,ds=\frac{1}{2\pi}\int_{0}^\pi 2f(2\cos(s))\cos(\ux s)\,ds.\]
The lemma follows by the change of variables $t=2\cos(s)$.
\end{proof}

Note that $A_{\mathbb{Z}^2}=A_{\mathbb{Z}}\otimes I  + I\otimes A_{\mathbb{Z}}$. Thus, the next lemma follows.

\begin{lemma}
For $(\ux,\uy)\in \mathbb{Z}^2$, let $\varrho_{\ux,\uy}=d_{\ux}*d_{\uy}.$ 
Then for any $(x,y)\in \mathbb{Z}^2$, and any continuous function $f:\mathbb{R}\to\mathbb{R}$, we have 
\[\langle  f(A_{\mathbb{Z}})e_{(x,y)},e_{(x+\ux,y+\uy)} \rangle=\int f(t) \varrho_{\ux,\uy}(t)dt.\]
\end{lemma}

As a direct consequence of the formula $\varrho_{\ux,\uy}=d_{\ux}*d_{\uy}$, we have 
\begin{equation}\label{varrhouvformula}
    \varrho_{\ux,\uy}(\lambda)=\int \frac{1}{\pi\sqrt{4-\theta^2}}\frac{1}{\pi\sqrt{4-(\lambda-\theta)^2}}
    \frac{1}4\sum_{\alpha=\pm \cos^{-1}\left(\frac{\theta}2\right)} \sum_{\beta=\pm \cos^{-1}\left(\frac{\lambda-\theta}2\right)}\exp\left(i(\alpha \ux+\beta \uy)\right)d\theta,
\end{equation}
where the integration is over all $\theta$ such that $|\theta|,|\lambda-\theta|\le 2$. 
\subsection{Complex Gaussian processes and the Gaussian wave $Z_\lambda$}\label{subsecwave}
Let $H$ be a countable set and let $M$ be a positive semidefinite matrix indexed with $H$. The (circularly-symmetric) complex Gaussian process with covariance matrix $M$ is random vector $Z\in\mathbb{C}^H$ such that for any vector $q\in \mathbb{C}^H$ with finite support,  $\langle Z,q \rangle$
is a complex Gaussian variable with mean zero and variance $\langle q,Mq \rangle$, that is, $\Re \langle Z,q \rangle$ and $\Im\langle Z,q \rangle$ are independent Gaussian random variables  with mean zero and variance $\tfrac{1}2\langle q,Mq \rangle$.

We define the matrix $M=M_\lambda$ indexed by $\mathbb{Z}^2\times \mathbb{Z}^2$ as
\begin{equation}\label{eqnewMdef}M((x+\ux,y+\uy),(x,y))=\frac{\varrho_{\ux,\uy}(\lambda)}{\varrho_{0,0}(\lambda)}.\end{equation}

The Gaussian wave $Z_\lambda$ is a complex Gaussian process with covariance matrix $M$.


\section{The spectrum of the discrete torus}

The most important ingredient in Theorem \ref{t:main} is the regularity of the eigenvalue distribution of the discrete torus. Uniform boundary conditions are needed because without randomization, the required regularity is an open problem. It translates to Conjecture \ref{c:roots}  of number theoretic nature below. 

For $b\in [0,1]$ let us consider the $n\times n$ Hermitian matrix 
$A_b$ with 
$$
(A_b)_{jk}=\begin{cases}
1 \qquad & |j-k|=1,
\\
 e^{\pm 2\pi i nb} &(j,k)=(n,1) \mbox{ resp. } (j,k)=(1,n),
 \\
0 &\mbox{otherwise.}\end{cases}
$$

In  words, $A_b$ is the adjacency matrix of the weighted graph given by the {\bf cycle  with boundary condition $b$}. Edges  $(j,j+1)$ have weight $1$, and the edge $(n,1)$ has weight $e^{\pm 2 \pi i n b}$ depending on  orientation. The following lemma is checked by direct computation. 

\begin{lemma}
Let  
\[v_k=\left(\exp\left(2\pi i j\left(\tfrac{k}{n}+b\right)\right)\right)_{j=1}^n.\]
Then $(v_k)_{k=0}^{n-1}$ form an orthogonal basis of $\mathbb{C}^n$, and
\[A_bv_k=2\cos\left(2\pi  \left(\tfrac{k}{n}+b\right)\right) v_k.\]
\end{lemma}
Let $P_k$ the orthogonal projection to the one dimensional subspace of $\mathbb{C}^n$ spanned by $v_k$. Then
\[P_k(\ell,j)=\frac{1}{n}\exp\left(2\pi i (\ell-j)\left(\tfrac{k}{n}+b\right)\right).\]

Let $e_1,e_2,\dots,e_n$ be the standard basis of $\mathbb{C}^n$.
Let $1\le j\le n$, and let $\ux$ be an integer such that $1\le j+\ux\le n$. Let $\mu_{b,\ux}=\mu_{n,b,\ux}$ be the spectral measure corresponding to $A_b$ and the pair of vectors $e_j$ and $e_{j+\ux}$. In other words, $\mu_{b,\ux}$ is the unique complex valued measure such that that for any continuous $f$, we have
\[\langle f(A_b)e_j,e_{j+\ux} \rangle=\int f(t)d\mu_{b,\ux}(dt).\]
It follows that 
\begin{align*}
\mu_{b,\ux}&=\frac{1}{n}\sum_{k=0}^{n-1} P_k(j+\ux,j)\delta\left(2\cos\left(2\pi\left(\tfrac{k}{n}+b\right)\right)\right)
\\&=\frac{1}{n}\sum_{k=0}^{n-1}\exp\left(2\pi i\ux\left(\tfrac{k}{n}+b\right)\right) \delta\left(2\cos\left(2\pi\left(\tfrac{k}{n}+b\right)\right)\right).
\end{align*}

Note that $\mu_{b,\ux}$ does not depend on $j$. 

The Cartesian product of two cycles with boundary conditions ${\rm{c}},{\rm{d}}$ is called the {\bf discrete torus with boundary conditions ${\rm{c}},{\rm{d}}$}. Its adjacency matrix has the tensor product expression
$A_{{\rm{c}},{\rm{d}}}=A_{\rm{c}}\otimes I + I\otimes A_{\rm{d}}$. 

The tensor product structure implies that the  
spectral measure of $A_{{\rm{c}},{\rm{d}}}$ at the pair of vectors $e_j\otimes e_k$ and $e_{j+\ux} \otimes e_{k+\uy}$ is given by the convolution 
\[\mu_{{\rm{c}},{\rm{d}},\ux,\uy}=\mu_{{\rm{c}},\ux}*\mu_{{\rm{d}},\uy}.\]

We need to show that this measure is sufficiently regular at most $\lambda$. Unfortunately, it is an open question for zero boundary conditions. In this case $\mu_{0,0,0,0}$ is the law of $2\Re(\xi_1+\xi_2)$, where $\xi_i$ are independent, uniformly chosen $n$-th roots of unities. Average regularity of this measure has to do with the mass near zero of the convolution of $\mu$ by its reflection. The main obstacle for proving a version of Theorem 1 for the ordinary discrete torus is the following.

\begin{conjecture}\label{c:roots}
$$
\mathbb{P}(|\Re (\xi_1+\xi_2+\xi_3+\xi_4)|<1/n^2)=1/n^{2+o(1)},\qquad n\to \infty.
$$\end{conjecture}

We can prove similar regularity under uniform boundary conditions, which behave nicely. For example, we will see in the proof of Proposition \ref{p:regularity} that for $C$ uniform on $[0,1]$ the averaged measure $\ev\mu_{C,0}$ is the arcsine law $\mathfrak{s}_2$, see \eqref{e:arcsine2}.

\subsection{A few continuity results}
Consider $\pi$ times the  Cauchy distribution
\begin{equation}\label{kappaetadef}
\kappa=\kappa_\eta=\frac\eta{x^2+\eta^2}dx. 
\end{equation}
For the proof of the next lemma see the appendix.
\begin{lemma}\label{newcontinuity} Let $\epsilon\in(0,1)$ and $p>1$. Then for $\lambda\in [3\epsilon-4,-3\epsilon]\cup[3\epsilon,4-3\epsilon]$, $\eta\ge 0$, $\ux,\uy\in \mathbb{Z}$ and $c_{\epsilon,p}=\epsilon^{-1/p}+c_p\epsilon^{-3/2}$ we have 
\begin{align*}
|(d_\ux*d_\uy)'(\lambda)|
&\le 2\varepsilon^{-3/2}(\ux^2+\uy^2+1)\text{ and}\\
|(\kappa_\eta-\pi \delta) *d_\ux*d_\uy(\lambda)|
&\le c_{\epsilon,p} (\ux^2+\uy^2+1) \eta^{1/p}.
\end{align*}
Moreover, for any $\ell>0$ such that $[E\pm \ell]\subset [3\epsilon-4,-3\epsilon]\cup[3\epsilon,4-3\epsilon]$, we have
$$
|\kappa_\eta *d_\ux*d_\uy(\lambda)-\pi d_\ux*d_\uy(E)|
\le c_{\epsilon,p} (\ux^2+\uy^2+1) (\eta^{1/p}+\ell)\qquad\text{ for all }\lambda\in [E\pm \ell].
$$

\end{lemma}
\subsection{Logarithmic singularity of $\varrho_{0,0}$}

Note that $\varrho_{0,0}$ has a singularity at $0$, but we have a good control over it as the next lemma shows. The proof is deferred to the appendix.

\begin{lemma}\label{varrhosing}
We have
$\varrho_{0,0}(x)\le \mathbbm{1}(|x|\le 4) \log(100/|x|).$
\end{lemma}

\subsection{A deterministic estimate on $\mu_{{\rm{c}},0}*\mu_{{\rm{d}},0}$}

\begin{lemma}\label{detest}
For any $({\rm{c}},{\rm{d}})\in [0,1]^2$ and an interval $[a,b]$, we have
\[\mu_{{\rm{c}},0}*\mu_{{\rm{d}},0}([a,b])\le \int_{a-\frac{8\pi}n}^{b+\frac{8\pi}n} \varrho_{0,0}(\lambda)d\lambda.\]
\end{lemma}
\begin{proof}

Note that $\mu_{{\rm{c}},0}=\mu_{{\rm{c}}+1/n,0}$. So 
we may assume that $({\rm{c}},{\rm{d}})\in \left[0,1/n\right]^2$. For the same reason, 
$\mu_{C,0}\times \mu_{D,0}$ has the same law no matter whether  $(C,D)$ is uniform random on $[0,1]^2$ or on $[0,1/n]^2$. Assume the latter. 
Since $\cos(x)$ has Lipschitz constant $1$, we see that if
\begin{align*}2\cos\left(2\pi\left({k}/n+{\rm{c}}\right)\right)+2\cos\left(2\pi\left({j}/n+{\rm{d}}\right)\right)\,&\in\, [a,b],
\qquad
\mbox{ then}
\\
2\cos\left(2\pi\left({k}/n+C\right)\right)+2\cos\left(2\pi\left({j}/n+D\right)\right)\,&\in\, \left[a-{8\pi}/n,b+{8\pi}/n\right].
\end{align*}
So 
$\mu_{\rm {c},0}*\mu_{{\rm{d}},0}([a,b])\le
\mu_{\rm {C},0}*\mu_{\rm{C},0}
(\left[a-{8\pi}/n,b+{8\pi}/n\right])$. We conclude by taking expectations and using Lemma~\ref{expdensity}. 
\end{proof}

\begin{lemma}\label{detest2}
Let $0<\varepsilon<2/3$. Assume that $\left[E_0-r-\frac{8\pi}n,E_0+r+\frac{8\pi}n\right]$ is contained in\break $[3\epsilon-4,-3\epsilon]\cup[3\epsilon,4-3\epsilon]$. Then
for any $({\rm{c}},{\rm{d}})\in [0,1]^2$
\begin{align*}\mu_{{\rm{c}},0}*\mu_{{\rm{d}},0}([E_0-r,E_0+r])&\le 2r\varrho_{0,0}(E_0)+\frac{16\pi}n  \varrho_{0,0}(E_0)+O_\varepsilon\left(\left(r+\frac{8\pi}n\right)^2\right),
\\\mu_{{\rm{c}},0}*\mu_{{\rm{d}},0}([E_0-r,E_0+r])&\ge 2r\varrho_{0,0}(E_0)-\frac{16\pi}n  \varrho_{0,0}(E_0)+O_\varepsilon\left(\left(r-\frac{8\pi}n\right)^2\right).
\end{align*}
\end{lemma}
\begin{proof}
Using Lemma \ref{newcontinuity}, we see that
\begin{align*}
\int_{E_0-r-\frac{8\pi}n}^{E_0+r+\frac{8\pi}n} \varrho_{0,0}(\lambda)d\lambda&\le \int_{E_0-r-\frac{8\pi}n}^{E_0+r+\frac{8\pi}n} \varrho_{0,0}(E_0)+O_\varepsilon(|\lambda-E_0|)d\lambda\\&=2r\varrho_{0,0}(E_0)+\frac{16\pi}n  \varrho_{0,0}(E_0)+O_\varepsilon\left(\left(r+\frac{8\pi}n\right)^2\right).
\end{align*}
Combining this with Lemma \ref{detest} the first statement follows. The second statement can be proved using the same method.
\end{proof}





\subsection{Regularity -- the proof of Theorem~\ref{torusthm1}}

This section contains the main estimate on the regularity of  the spectral measure of the discrete torus with uniform boundary conditions. We show that the measure is close to its mean, 
the spectral measure of $A_{\mathbb Z^2}$. 


 

Observe that the expectation of $\mu_{C,D,\ux,\uy}$ is  given by the spectral measure of $\mathbb{Z}^2$ which was defined in Section \ref{subsecspec}.
\begin{lemma}\label{expdensity}
Let $C$ and $D$ be independent uniform random elements of $[0,1]$. Then $\mathbb{E}\mu_{C,\ux}$ has density $d_\ux$, and
$\mathbb{E}\mu_{C,D,\ux,\uy}=(\mathbb{E}\mu_{C,\ux})*(\mathbb{E}\mu_{D,\uy})$ has density $\varrho_{\ux,\uy}$.
\end{lemma}

First we reduce  Theorem \ref{torusthm1} to a few key estimates. Recall the $\kappa=\kappa_\eta$ was defined in \eqref{kappaetadef} as $\pi$ times the  Cauchy distribution.

By expanding the variance, the left hand side of \eqref{e:torus} can be written as
\[\mathbb{E}
\|\mu_{C,D,\ux,\uy}*\kappa
-\mathbb{E}
\mu_{C,D,\ux,\uy}*\kappa\|_2^2
=\mathbb{E}
\|\mu_{C,\ux}*\mu_{D,\uy}*\kappa
-(\mathbb{E}\mu_{C,\ux})*(\mathbb{E}\mu_{D,\uy})*\kappa\|_2^2.\]
Here we use the convention that if  a measure $\nu$ has density, then $\|\nu\|_2$ denotes its  $L^2$-norm. For the equality, we used that $\mu$ is a convolution, the linearity of expectation, and that $\mu_{C,\ux}$ and $\mu_{D,\uy}$ are independent. 

Next, for a measure $\nu$ with density let $\nu(0)$ mean the density at $0$. 
Define $\nu'$ by setting $\nu'(A)=\overline{\nu(-A)}$ for any measurable set $A$. Using that $\|\nu\|_2^2=(\nu*\nu')(0)$
and expanding by linearity, we write the left hand side of \eqref{e:torus}  as
\begin{equation}\label{varianceEQ}
 \left(\mathbb{E}[\mu_{C,\ux}*\mu_{C,\ux}']*\kappa*\mathbb{E}[\mu_{D,\uy}*\mu_{D,\uy}']*\kappa-\mathbb{E}(\mu_{C,\ux})*\mathbb{E}(\mu_{C,\ux}')*\kappa*\mathbb{E}(\mu_{D,\uy})*\mathbb{E}(\mu_{D,\uy}')*\kappa\right)(0).\end{equation}
Using Proposition~\ref{p:regularity} below, we see that
\begin{equation}\label{EmmEmEm}
    \mathbb{E}[\mu_{C,\ux}*\mu_{C,\ux}']*\kappa=\left(\mathbb{E}(\mu_{C,\ux})*\mathbb{E}(\mu_{C,\ux}')+\frac{\delta_0}n\right)*\kappa+\tau_\ux,\text{ where }\|\tau_\ux\|_2\le c\frac{1+|\ux|}{\eta^\epsilon \sqrt{n}}.
\end{equation}

 Inserting \eqref{EmmEmEm} into \eqref{varianceEQ}, we see that \eqref{varianceEQ} can be written as
\begin{multline}
    \frac{(\kappa*\kappa)(0)}{n^2}+\frac{\left(\left(\mathbb{E}(\mu_{C,\ux})*\mathbb{E}(\mu_{C,\ux}')+\mathbb{E}(\mu_{D,\uy})*\mathbb{E}(\mu_{D,\uy}')\right)*\kappa*\kappa\right)(0)}{n}\\+(\tau_\ux*\mathbb{E}[\mu_{D,\uy}*\mu_{D,\uy}']*\kappa)(0) +(\tau_\uy*\mathbb{E}[\mu_{C,\ux}*\mu_{C,\ux}']*\kappa)(0)-(\tau_\ux*\tau_\uy)(0).
\end{multline}

Here the first term is
\[\frac{(\kappa*\kappa)(0)}{n^2}=\frac{1}{n^2}\int \left(\frac\eta{t^2+\eta^2}\right)^2dt=\frac{\pi}{2n^2\eta} .\]

Observing that $|\mathbb{E}(\mu_{C,\ux})|=|\mathbb{E}(\mu_{C,\uy}')|\le d_0(x)dx$, we can bound the second term as
\begin{align*}
    \left|\frac{\left(\left(\mathbb{E}(\mu_{C,\ux})*\mathbb{E}(\mu_{C,\ux}')+\mathbb{E}(\mu_{D,\uy})*\mathbb{E}(\mu_{D,\uy}')\right)*\kappa*\kappa\right)(0)}{n}\right|&\le \frac{2(\varrho_{0,0}*(\kappa*\kappa))(0)}n.
\end{align*}
Using  Lemma \ref{varrhosing}, the right hand side is at most
    \[\frac{1}n O\left(\int_{-4}^4 \frac{-\eta\log(t/ 100)}{t^2+4\eta^2} dt\right)=\frac{1}n O\left(
    \int_{-4/\eta} ^{4/\eta}  \frac{-\log(x\eta/ 100)}{x^2+4} dx\right)=\frac{O(\max(1,-\log(\eta))}n<\frac{1}{\eta^\varepsilon\sqrt{n}}
\]
for any large enough $n$ by our assumptions on $\eta$ and $\varepsilon$.

By Cauchy-Schwarz we have $|(\nu_1*\nu_2)(0)|\le \|\nu_1\|_2\|\nu_2\|_2$. Combining this with  Proposition~\ref{p:regularity}, we can bound the last three term as follows
\begin{align*}\Big|\big(\tau_\ux*\mathbb{E}&[\mu_{D,\uy}*\mu_{D,\uy}']*\kappa)(0) +(\tau_\uy*\mathbb{E}[\mu_{C,\ux}*\mu_{C,\ux}']*\kappa\big)(0)-(\tau_\ux*\tau_\uy)(0)\Big|\\&\le \|\tau_\ux\|_2\|\mathbb{E}[\mu_{D,\uy}*\mu_{D,\uy}']*\kappa\|_2 +\|\tau_\uy\|_2\|\mathbb{E}[\mu_{C,\ux}*\mu_{C,\ux}']*\kappa\|_2+\|\tau_\ux\|_2\|\tau_\uy\|_2\\&\le c^2 \frac{2+2|\ux|+|\uy|}{\eta^\varepsilon\sqrt{n}},\end{align*}
where in the last step we used the conditions on $\eta$ and $(\ux,\uy)$ to see that $\|\tau_\uy\|_2\le c$.
Thus  Theorem \ref{torusthm1} reduces to the following. Let $\mu=\mu_{C,\ux}$.
\begin{proposition}\label{p:regularity} For every $\tfrac{1}4>\epsilon>0$ there is $c>0$ so that for all $n$, $n^2\ge \eta\ge n^{-2}$ and $|\ux|\le n^{1/2-2\varepsilon}$ we have
\begin{align}\label{torusineq1}
\left\|\big(\mathbb{E}[\mu *\mu ']-\frac{1}{n}\delta_0-(\mathbb{E}\mu )*(\mathbb{E}\mu )'\big)*\kappa\right\|_2\
&\le c\frac{1+|\ux|}{\eta^\epsilon \sqrt{n}},\\ \label{torusineq3}
 \|(\mathbb{E}\mu )*(\mathbb{E}\mu )'*\kappa\|_2&\le c,\\\label{torusineq2}
 \|\mathbb{E}[\mu *\mu ']*\kappa\|_2&\le c.
\end{align}

\end{proposition}
\begin{proof}
Recall that
\begin{equation}\label{e:muc}\mu=\frac{1}{n}\sum_{j=0}^{n-1}\exp\left(2\pi i\ux\left(\tfrac{j}{n}+C\right)\right) \delta\left(2\cos\left(2\pi\left(\tfrac{j}{n}+C\right)\right)\right).\end{equation}
Therefore, by the product formula for the difference of cosines, 
\begin{align*}\mu*&\mu'
=\frac{1}{n^2}\sum_{j=0}^{n-1}\sum_{k=0}^{n-1} \exp\left(2\pi i \ux \tfrac{j-k}{n} \right)\delta\Big(4\sin\left(2\pi\tfrac{k-j}{2n}\right)\sin\left(2\pi\left(\tfrac{j+k}{2n}+C\right)\right)
\Big).
\end{align*}

For $b>0$, let $\mathfrak{s}_b$ be the arcsine law on $[-b,b]$ with density
\begin{equation}\label{e:arcsine2}f_b(x)=\frac{\mathbbm{1}(|x|<b)}{\pi\sqrt{b^2-x^2}}
\end{equation}
and we set $\mathfrak{s}_0=\delta_0$. 
Since $C$ is a uniform element of $[0,1]$, we have
\[\mathbb{E} \delta\Big(4\sin\left(2\pi\tfrac{k-j}{2n}\right)\sin\left(2\pi\left(\tfrac{j+k}{2n}+C\right)\right)\Big)=\mathfrak{s}_{h_{k-j}},\qquad h_j=4\left|\,\sin\left(\pi\tfrac{j}{n}\right)\right|.\]

Therefore,
\[\mathbb{E}[\mu*\mu']=\frac{1}{n^2}\sum_{j=0}^{n-1}\sum_{k=0}^{n-1} \exp\left(2\pi i \ux \tfrac{j-k}{n} \right)\mathfrak{s}_{h_{k-j}}
=\frac{1}{n}\sum_{d=0}^{n-1} \exp\left(-2\pi i \ux \tfrac{d}{n} \right)\mathfrak{s}_{h_{d}},
\]
where we used  the fact  that $h_d=h_{d'}$ if $d\equiv d'\mod{n}$.
Assume that $n$ is even; the odd case is proven similarly. By pairing the terms $d$ and $n-d$  using $h_d=h_{n-d}$, we conclude that 
\begin{align*}
\mathbb{E}&[\mu*\mu']
=
\frac{1}{n}\delta_0+\frac{(-1)^\ux}{n} \mathfrak{s}_4+f(x)\,dx,\qquad
f=\frac{2}{n}\sum_{d=1}^{n/2-1}\cos\left(\pi \ux\tfrac{2d}{n}\right)f_{h_d}.
\end{align*}

Since each term in \eqref{e:muc} has the same expectation, we can integrate instead summing:
\begin{equation}\label{e:emu}\mathbb{E}\mu=\mathbb{E}\int_0^1\exp\left(2\pi i\ux\left(t+C\right)\right) \delta\left(2\cos\left(2\pi\left(t+C\right)\right)\right)\,dt.
\end{equation}
but now the integral does not depend on $C$. Substituting $C=0$  we see that $\ev \mu_{C,0}=\mathfrak{s}_2$. Also, we may drop the expectation from the right hand side of \eqref{e:emu}. Then, repeating the argument above with integrals instead of sums, we get 
$$
(\ev\mu)*(\ev\mu)'=g(x)dx, \qquad g=\int_0^1 \cos\left(\pi \ux t\right)f_{h_{nt/2}}dt.
$$

We see that the density of $\mathbb{E}[\mu *\mu ']-\frac{1}{n}\delta_0-(\mathbb{E}\mu )*(\mathbb{E}\mu )'$ is given by 
\[f-g+\frac{(-1)^\ux}{n}f_4.\]

The function $f$ is close to a Riemann sum for the integral in $g$. For each $x>0$, the support of the initial few $f_{h_d}$ may exclude $x$. Consider the unique $j$ so that $h_{j-1}\le x<h_{j}$. 
Let $r(x)=\int_0^{2j/n}f_{h_{nt/2}}(x)dt$. Then with $\varphi(t)=f_{h_{nt/2}}(x)$ and $s(t)=\cos (\pi \ux t)\varphi(t)$, we have
$$
|g(x)-f(x)|\le r(x) + \left|\int_{2j/n}^{1} s(t)\,dt -\frac{2}{n}\sum_{d=j}^{n/2-1} s(\tfrac{2d}{n})\right|.
$$
The Riemann sum error is bounded by $2/n$ times the total variation of $s$: 
$$
\int_{2j/n}^1 |s'(t)| \, dt = \int_{2j/n}^1 |\pi \ux \sin(\pi \ux t) \varphi(t) - \cos(\pi \ux t) \varphi'(t)|\,dt.
$$
Since $\varphi(t)\ge 0$ is decreasing in $t$, we can bound the integrand by $\pi|\ux|\varphi(t)-\varphi'(t)$ to get 
$$\int_{2j/n}^1 |s'(t)| \, dt\le (\pi|\ux|+1)\max \varphi=(\pi|\ux|+1)f_{h_{j}}(x).
$$
Let  
\begin{align*}\bar f(x)=f_{h_{j}}(x)= \frac{1}{\pi\sqrt{(h_{j}+x)(h_{j}-x)}}\le \frac{1}{\sqrt{h_{j}}}\frac{1}{\sqrt{h_{j}-x}}.
\end{align*}
So for $p\in [1,2)$ we have
$$
\|\bar f\|_p^p\le \sum_{j=1}^{n/2} h_{j}^{-p/2} \int_{h_{j-1}}^{h_{j}}(h_{j}-x)^{-p/2}dx= \frac{1}{1-p/2}\sum_{j=1}^{n/2} \frac{(h_{j}-h_{j-1})^{1-p/2}}{h_{j}^{p/2}}.
$$
The sum approximates
$$
\frac{n}2\int_0^1  \frac{(\tfrac{2}{n}\partial_t (h_{tn/2}))^{1-p/2}}{h_{tn/2}^{p/2}}\,dt=c_p n^{p/2},\qquad \mbox {so }
\|\bar f\|_p\le c_p \sqrt{n}.$$ Next, we  bound $r(x)$. With $t_0$ solving $h_{t_0n/2}=4\sin(\pi t_0/2)=x$  we have 
$$
r(x)=\int_{t_0}^{2j/n}\frac{1}{4\pi\sqrt{ \sin(\pi t/2)^2-\sin(\pi t_0/2)^2}}\,dt.
$$
Using the concavity of $\sin(\pi t/2)$ on $[ t_0,1]$, we see that for $t\in [t_0,1]$, we have
\[\sin(\pi t/2)\ge \sin(\pi t_0/2)+q(x)(t-t_0),\text{ where }q(x)=\frac{1-\sin(\pi t_0/2)}{1-t_0}\ge 1-x/4.\]

Thus, the expression under the square root is at least $ \tfrac{x}2(1-x/4)(t-t_0)$.

Using the fact that $2j/n -t_0<2/n$, we see that
\[r(x)\le \int_{0}^{2/n} \frac{1}{4\pi\sqrt{\tfrac{x}2(1-x/4)w}}dw= \frac{c}{\sqrt{n x(4-x)}},\qquad \|r\|_p\le c_p/\sqrt{n}.\]


Clearly, $\|f_4\|_p$ is finite for all $p\in [1,2)$. Choosing $p,q\in [1,2)$ such that  $1+\frac{1}{2}=\frac{1}p+\frac{1}{q}$, we can use Young's convolution inequality to obtain  following upper bound for the $L^2$-norm of the density of $\left(\mathbb{E}[\mu*\mu']-(\mathbb{E}\mu)*(\mathbb{E}\mu)'\right)*\kappa $: 
\begin{align*}
\left\|\left(f-g+\frac{(-1)^\ux}{n}f_4\right)*\kappa\right\|_2&\le \left\|f-g+\frac{(-1)^\ux}{n}f_4\right\|_p \|\kappa\|_q\\&\le \left(\frac{4}n(1+\pi|u|)\|\bar{f}\|_p+2\|r\|_p+\frac{1}n \|f_4\|_p\right)\|\kappa\|_q.      
\end{align*}

 By scaling, we have $\|\kappa\|_q\le c_q \eta^{1/q-1}$. Thus, choosing $p$ and $q$ such that $1/q=1-\varepsilon$, we get the bound \eqref{torusineq1}.


By taking absolute values, we see that the norms in $\eqref{torusineq2}, \eqref{torusineq3}$ are  maximized when $\ux=0$. Assume $\ux=0$, so that $\ev\mu=\mathfrak{s}_2$. Then 
$$
\|\ev \mu*\ev\mu*\kappa\|_2=\|\mathfrak{s}_2*\mathfrak{s}_2*\kappa\|_2\le \|\mathfrak{s}_2*\mathfrak{s}_2\|_2\|\kappa\|_1 \le \pi \|\mathfrak{s}_2\|_{4/3}^2
$$
by two uses of Young's inequality. Since the density of $\mathfrak{s}_2$ has $1/\sqrt{x}$ singularities, the last quantity is finite, showing \eqref{torusineq3}.
The last inequality \eqref{torusineq2} is a consequence of the first two.
\end{proof}

\subsection{Close pairs}\label{secclosepairs}

In this section we show the following estimate on the number of torus eigenvalue pairs that are $n^{-2}$-close.  Let $(C,D)$ be a uniform random element of $[0,1]^2$.

\begin{proposition}\label{Lemmaclosepairs}
Let $E\in(0,4)$, and let  $r=r_n$ satisfy  $\frac{\log n}{n}\le r\le o(1)$ and  $nr\in \mathbb Z$.
Partition $[E-r,E+r]$ into a set $\mathcal I$ of intervals of length $n^{-2}$. Then 
\[\sum_{J\in \mathcal{I}} \mathbb{E}\big(\mu_{C,0}*\mu_{D,0}(J)\big)^2=O(n^2r).\]
\end{proposition}


Let  $(F,G)$ be a uniform random element of $[0,r]^2$. Since $\mu_{c+1/n,0}=\mu_{c,0}$, we see that $\mu_{F,0}*\mu_{G,0}$
and $\mu_{C,0}*\mu_{D,0}$ have the same law. 

Let $  S=(-n/2,n/2]\cap \mathbb Z$. Then
\[\sum_{J\in \mathcal{I}} \mathbb{E}\big(\mu_{C,0}*\mu_{D,0}(J)\big)^2=\sum_{J\in \mathcal{I}} \mathbb{E} \left|\{(j,k)\in S^2\,:\,\lambda(j,k,F,G)\in J\}\right|^2.\]

Let
\[H=\left\{(j,k)\in S^2\,:\, \left| \lambda(j,k,0,0)-E\right|\le (1+8\pi)r\right\}.\]

Note that if $\lambda(j,k,f,g)\in [E-r,E+r]$, for some $(j,k)\in S^2$ and $f,g\in [0,r]$, then $(j,k)\in H$. Thus,
\begin{multline*}\sum_{J\in \mathcal{I}} \mathbb{E} \left|\{(j,k)\in S^2\,:\,\lambda(j,k,F,G)\in J\}\right|^2\\\le \mathbb{E} \left|\left\{(j,k)\in H,(\ell,m)\in S^2\,:\: \left|\lambda(j,k,F,g)-\lambda(\ell,m,F,g)\right|\le {1}/{n^2}\right\}\right|.
\end{multline*}

Assume that given $(j,k)\in H$, $\ell\in S$ and $g\in [0,r]$, we can provide a good upper bound on
\begin{equation}\label{Emgivenjkl}
    \mathbb{E}\left|\left\{m\in S\,:\, \left|\lambda(j,k,F,g)-\lambda(\ell,m,F,g)\right|\le {1}/{n^2}\right\}\right|,
\end{equation}
which does not depend on $g$. Summing these bounds over the choice of $(j,k)$ and $\ell$, we could obtain an upper bound on $\sum_{J\in \mathcal{I}} \mathbb{E}\big(\mu_{C,0}*\mu_{D,0}(J)\big)^2$.

Lemma~\ref{jlbecs} below provides a good estimate on the expectation in \eqref{Emgivenjkl}, but only under the assumption that $\cos\left(2\pi{j}/n\right)$ and $\cos\left(2\pi{\ell}/n\right)$ are both large. More precisely, write $E=E_1+E_2$, with $0<E_1<E_2<2$, and define the set
\[L=\left\{j\in S\,:\,2\cos\left(2\pi{j}/n\right)\ge E_1\right\}.\]
Then Lemma~\ref{jlbecs} gives an estimate on the expectation in \eqref{Emgivenjkl} under additional  the assumption that $j,\ell\in L$.

Despite the fact that  Lemma~\ref{jlbecs} only proved for $k,\ell\in L$, we can still use it to provide a bound for $\sum_{J\in \mathcal{I}} \mathbb{E}\big(\mu_{C,0}*\mu_{D,0}(J)\big)^2$. The reason for this is that for $(j,k)\in H$, we have $j\in L$ or $k\in L$ provided that $n$ is large enough. See the proof of Proposition~\ref{Lemmaclosepairs} for details.

\begin{lemma}\label{jlbecs}
 Let $j,\ell\in L$ and $(j,k)\in H$. Fix $g\in [0,r]$ and choose $F$ uniformly at random from $[0,r]$. Then with constants depending on  $E$ and $E_1$ only, 
 \[\sum_{m\in S}\mathbb{P}\left(\left|\lambda(j,k,F,g)-\lambda(\ell,m,F,g)\right|\le {1}/{n^2}\right)=\begin{cases}
 O\left(\frac{1}n\right)+O\left(\frac{1}{nr|j-\ell|}\right)&\text{if }j\neq \ell,\\
 O(1)&\text{if }j=\ell.
 \end{cases}\]
\end{lemma}
\begin{proof}[Proof of Lemma \ref{jlbecs}]
The set $L$  is set up so that for  $j,\ell\in L$, 
\begin{equation}\label{e:jL}
|j|\le \frac{n}4, \qquad \frac{E_1}2\le \cos\left(\pi\frac{j+\ell}{n}\right)\le 1, \qquad   \frac{|j-\ell|}{n} \le \left|\sin\left(\pi\frac{j-\ell}{n}\right)\right|\le \pi \frac{|j-\ell|}{n}.   
\end{equation}

 We start by the case $j\neq \ell$. Let
\[
u(f)=2\cos\left(2\pi\left({j}/n+f\right)\right)+2\cos\left(2\pi\left({k}/n+g\right)\right)-2\cos\left(2\pi\left({\ell}/n+f\right)\right).
\]
We have $\left|\lambda(j,k,f,g)-\lambda(\ell,m,f,g)\right|\le \frac{1}{n^2}$ if and only if $u(f)\in I'$ with  
\[ I'=\left[2\cos\left(2\pi\left({m}/n+g\right)\right)-{1}/{n^2},2\cos\left(2\pi\left({m}/n+g\right)\right)+{1}/{n^2}\right]
\]
and for $f\in [0,r]$  when $u(f) \in I'$ we also have 
$$2\cos\left(2\pi \left(\frac{m}{n}+g\right)\right)\in I, \qquad 
I=\left[\min(u(0),u(r))-{1}/{n^2},\max(u(0),u(r))+{1}/{n^2}\right].$$
Let $M\subset S$ be the set of  indices $m$ satisfying the above condition. 

By the product formula for $\sin$, we have
\begin{align*}u'(f)=-8\pi \sin\left(\pi \frac{j-\ell}{n}\right)\cos\left(\pi\left(\frac{j+\ell}{n}+2f\right)\right).
\end{align*}
We use  \eqref{e:jL} and  $f\in [0,r]$  to get the bound
\begin{align*}  \frac{|j-\ell|}{n}(4\pi E_1-O(r))\le  |u'(f)|\le 8\pi^2 \frac{|j-\ell|}{n}.
\end{align*}
Using the change of variables formula, we see that 
\begin{equation}\label{PAm}
\mathbb{P}(u(F)\in I')\le \frac{|I'|}{r \min_{f\in u^{-1}(I')} u'(f)}
 =O\left(\frac{1}{nr  |j-\ell|}\right).
\end{equation}
Next, we  estimate the size of the sets $M_1=M\cap [0,\infty)$, and $M_2=M\cap (-\infty, 0]$. 
Choose $E_3$ such that $\max(2-E,E_2)<E_3<2$. By the Lipschitz continuity of $\cos(\cdot)$, we see that for all $f\in [0,r]$, we have
\begin{equation}\label{eu(f)}
u(f)=2\cos\left(2\pi {j}/n\right)+2\cos\left(2\pi {k}/n\right)-2\cos\left(2\pi {\ell}/n\right)+O(r)
\end{equation}
 Using that $(j,k)\in H$ and $\ell\in L$, we see that
 \[2\cos\left(2\pi {j}/n\right)+2\cos\left(2\pi {k}/n\right)=E+O(r)\text{ and }2\cos\left(2\pi {\ell}/n\right)\ge E_1.\]
Combining these with \eqref{eu(f)}, we obtain
\begin{align*}
    u(f)\le E-E_1+O(r)=E_2+O(r)\le E_3-1/{n^2}
\end{align*}
for all large enough $n$, by the choice of $E_3$. Similarly, $u(f)\ge E-2-O(r)\ge 1/n^2-E_3$
and so 
$
    I\subset [-E_3,E_3]
$.
This and  $g\in[0,r]$ imply that $M_1$ is of the form $\{m_{\min},m_{\min}+1,\dots,m_{\max}\}$. For an $m$ such that $m,m+1\in M_1$, we have
\begin{equation}\label{mm1becs}2\cos\left(2\pi\left({m}/n+g\right)\right)-2\cos\left(2\pi\left((m+1)/n+g\right)\right)= \frac{4\pi}n \sin\left(2\pi\left({m'}/n+g\right)\right)
\end{equation}
for some $m'\in [m,m+1].$
Using $I\subset [-E_3,E_3]$, we see that $\left|2\cos\left(2\pi\left({m'}/n+g\right)\right)\right|\le E_3$.
Therefore, since $0\le m'\le n/2$, for any large enough $n$, we have
\[c\le \sin\left(2\pi\left({m'}/n+g\right)\right), \qquad 0<c=\sqrt{1-\left({E_3}/2\right)^2}.\]
Combining this with \eqref{mm1becs}, we see that
\begin{equation}\label{eqapart}
{4\pi c}/n\le 2\cos\left(2\pi\left({m}/n+g\right)\right)-2\cos\left(2\pi(m+1)/n+g\right).
\end{equation}
Adding these estimates over $m$ and using the definition of $M_1$, we see that
\begin{align*}
 (|M_1|-1) {4\pi c}/n &\le
2\cos\left(2\pi\left({m_{\min}}/n+g\right)\right)
-2\cos\left(2\pi\left({m_{\max}}/n+g\right)\right)\\&
\le |I|=|u(0)-u(r)|+{2}/{n^2}.
\end{align*}

Using our estimate on $u'(f)$, we see that 
$|u(0)-u(r)|\le8\pi^2  r |j-\ell|/n$.
Therefore,\break $|M_1|= O(r |j-\ell|+1).$
This and the identical estimate for $|M_2|$ combined with \eqref{PAm} gives the first claim. 

The $j=\ell$ case follows from a consequence of the proof of  \eqref{eqapart}: for every large enough $n$ there can be at most one $m\in S\setminus \{k\}$ such that 
\[\left|2\cos\left(2\pi\left({k}/n+g\right)\right)-2\cos\left(2\pi\left({m}/n+g\right)\right)\right|\le {1}/{n^2}.\qedhere\]
\end{proof}

\begin{proof}[Proof of Proposition \ref{Lemmaclosepairs}]
Let  $(F,G)$ be a uniform random element of $[0,r]^2$. Since $\mu_{c+1/n,0}=\mu_{c,0}$, we see that $\mu_{F,0}*\mu_{G,0}$
and $\mu_{C,0}*\mu_{D,0}$ have the same law. 
For $J\in\mathcal I$, write
\[
\mu_{F,0}*\mu_{G,0}(J)=Y_J(S,S), \qquad 
Y_J(U,V)=\sum_{j\in U}\sum_{k\in V} \mathbbm{1}\left(\lambda(j,k,F,G)\in J \right).\]
For large enough $n$, for all $(f,g)\in [0,r]^2$ and  $\lambda(j,k,f,g)\in [E-r,E+r]$,
we have $j\in L$ or $k\in L$. Thus for such $n$,
\begin{equation}
    \label{e:LSSL}
    Y_J(S,S)\le Y_J(L,S)+Y_J(S,L).
\end{equation}
Then  $Y_J(S,L)$ has the same law as $Y_J(L,S)$, and we have  
\begin{equation*}
\sum_{J\in \mathcal I}Y_J(L,S)^2 = \sum_{(j,k, \ell, m)\in (L\times S)^2}
\mathbbm{1}\left(\lambda(j,k,F,G),\ \lambda(\ell,m,F,G)\in J \mbox{ for some } J\in \mathcal I\right).
\end{equation*}
The indicated  event  implies  $(j,k), (\ell,m)\in H$. Taking expectations, we get 
\begin{equation}\label{triplesum}
\mathbb E \sum_{J\in \mathcal I}Y_J(L,S)^2 \le \sum_{\substack{(j,k)\in H\\j\in L}}\sum_{(\ell,m) \in L\times S} \mathbb{P}\left(\left|\lambda(j,k,F,G)-\lambda(\ell,m,F,G)\right|\le {1}/{n^2}\right).
\end{equation}
Lemma~\ref{jlbecs}  bounds the inner sum as
$$ O(1)+\sum_{\ell\in L\setminus\{j\}}\left( O\left({1}/n\right)+O\left({1}/(nr|j-\ell|)\right)\right)\le 
 O(1)+O\left(\log(n)/(nr)\right)= O(1).
$$
By  Lemma~\ref{detest2} we have 
 \[|H|=n^2\mu_{0,0}*\mu_{0,0}([E-(1+8\pi)r,E+(1+8\pi)r])=n^2 O(r). 
 \]
In \eqref{triplesum}, the first sum is over a set of size $n^2 O(r)$ and each term is $O(1)$. So $\mathbb E \sum_{J\in \mathcal I}Y_J(L,S)^2=n^2O(r)$.  The claim follows by Cauchy-Schwarz and \eqref{e:LSSL}.
\end{proof}

\section{Eigenvectors of GUE-perturbed locally convergent matrices}
In this section, we prove Theorem~\ref{lemmaquenched1A}. As preliminaries, we review the notion of the free convolution with the semicircle distribution, and the results of \cite{benigni2020eigenvectors}.

In Section \ref{secallpoint}, we prove Theorem~\ref{lemmaquenched1}, which is a simpler deterministic version of Theorem~\ref{lemmaquenched1A}, where it is assumed that the regularity conditions hold at every point of a given interval not just with high probability. To prove Theorem~\ref{lemmaquenched1A}, we subdivide the interval $[E\pm \ell/2]$ into smaller intervals. Combining the conditions of Theorem~\ref{lemmaquenched1A} and  the continuity properties of the Stieltjes-transform, we show that for most of these smaller intervals, Theorem~\ref{lemmaquenched1}  can be applied. Thus, the perturbed eigenvectors from these intervals have the desired local weak limit. Finally, using the uniform integrability condition and the norm estimate of Lemma~\ref{normbecs}, we prove that the intervals where Theorem~\ref{lemmaquenched1} can not be applied contain a negligible number of eigenvalues.


\subsection{The free convolution with the semicircle distribution}\label{freeconvprel}

The results of this section are taken from \cite{biane1997free}.

Let $\mu$ be a probability measure on $\mathbb{R}$ and $t>0$. In this section, we define the free convolution of $\mu$ and the semicircle distribution of variance $t$.

We define $v_t(\lambda)$ as follows
\[v_t(\lambda)=\inf\left\{v\ge 0\,:\,\int \frac{1}{(x-\lambda)^2+v^2}\mu (dx)\le \frac{1}{t} \right\}.\]

We also define
\[\psi_t(\lambda)=\lambda-t\int\frac{x-\lambda}{(x-\lambda)^2+v_t(\lambda)^2}\mu(dx).\]

\begin{lemma}
The map $\psi_t$ is a homeomorphism from $\mathbb{R}$ to $\mathbb{R}$.
\end{lemma}

Since $\psi_t$ is a homeomorphism, the inverse $\psi_t^{-1}$ of $\psi_t$ is well defined. Let us define
\begin{equation}\label{ptdef}p_t(\lambda)=\frac{v_t(\psi_t^{-1}(\lambda))}{\pi t}.
\end{equation}

\begin{lemma}
The function $p_t$ is a probability density function, that is, $p_t$ is non-negative and $\int p_t(x)dx=1$. 
\end{lemma}

The probability measure with density $p_t$ is called the \textbf{free convolution} of $\mu$ and the semicircle distribution of variance $t$. This is not the usual definition of the free convolution, but it will be convenient for our purposes.

Let $m$ be the Stieltjes transform of $p_t$, that is, for any $\lambda\in \mathbb{R}$ and $\eta>0$, let us define
\[m(\lambda+i\eta)=\int \frac{1}{x-(\lambda+i\eta)} p_t(x)dx.\]
Then $m$ can be continuously extended to the real line and for any $\lambda\in \mathbb{R}$, and we have
\begin{align}\label{reimfc1}\Im m(\lambda)&=\pi p_t(\lambda)=\frac{v_t(\psi_t^{-1}(\lambda))}t, \qquad
\Re m(\lambda)=\frac{1}{t}(\psi_t^{-1}(\lambda)-\lambda).
\end{align}

\subsection{The results of Benigni}

For all $n$, let $D=D_n=\text{diag}(d_1,d_2,\dots, d_n)$ be an $n\times n$ diagonal matrix with real entries. Let $m_D$ be the Stieltjes transform of the empirical probability measure on the diagonal entries of $D$, that is,
\[m_D(\lambda+i\eta)=\frac{1}n\sum_{j=1}^n \frac{1}{d_j-(\lambda+i\eta)}.\]

Let us choose $\eta_*=\eta_{*,n}$ and $r=r_n$ such that
\[n^{-1}\le \eta_*\le n^{-\varepsilon'}\text{ and }\eta_*n^{\varepsilon'}\le r \le n^{-\varepsilon'}\]
for some $\varepsilon'>0$.

Fix a $\lambda_0$. Assume that there are two constants $c>0$ and $c'>0$ independent of $n$ such that
\[c\le \Im m_D(\lambda+i\eta)\le c'\]
for all $D=D_n$, $\lambda\in [\lambda_0-r,\lambda_0+r]$ and $\eta_*\le \eta\le 10$.

Choose $t=t_n$ such that $t\in [\eta_*n^{\omega},n^{-\omega}r]$ for some $\omega>0$ independent of $n$.

Let $m_t=m_{t_n,n}$ be the Stieltjes transform of the free convolution of the empirical measure on the diagonal entries of $D$ and the semicircle distribution of variance $t$. Let $p_t=p_{t_n,n}$ be the density of this free convolution. We also define $v_t(\lambda)=v_{t_n,n}(\lambda)$ and the homeomorphism $\psi_t=\psi_{t_n,n}$ the same way as in  Section \ref{freeconvprel}, that is,
\[v_t(\lambda)=\inf\left\{v\ge 0\,:\, \frac{1}{n}\sum_{j=1}^n \frac{1}{(d_j-\lambda)^2+v^2}\le \frac{1}{t} \right\},\quad\psi_t(\lambda)=\lambda-\frac{t}{n}\sum_{j=1}^n \frac{d_j-\lambda}{(d_j-\lambda)^2+v_t(\lambda)^2}.\]

For $i=0,1,2,\dots,n$, we define the $i$th quantile $\gamma_{i,t}=\gamma_{i,t_n,n}$ of this measure by the equation
\begin{equation}\label{quantiledef}\int_{-\infty}^{\gamma_{i,t}} p_t(x)dx=\frac{i}{n}.
\end{equation}

For $q\in \mathbb{R}^n$ and $k=1,2,\dots, n$, we define
\[\sigma^2(q,k)=\sum_{j=1}^n\frac{|q_j|^2 t}{(d_j-\gamma_{k,t}-t\Re m_t(\gamma_{k,t}))^2+(t\Im m_t (\gamma_{k,t}))^2}, \]

Using Equations \eqref{reimfc1}, this can be rewritten as

\begin{equation}\label{sigmadef}\sigma^2(q,k)=\sum_{j=1}^n\frac{|q_j|^2 t}{(d_j-\psi_t^{-1}(\gamma_{k,t}))^2+v_t(\psi_t^{-1}(\gamma_{k,t}))^2}.
\end{equation}

We define $W_t=D+\sqrt{t} W$, where $W$ is an $n\times n$ GUE matrix defined in Section \ref{subsecGUE}.

Let $u_1,u_2,\dots,u_n$ be the $\ell^2$ normalized eigenvectors of $W_t$ with independent uniform random phases, we order these vectors such that the corresponding eigenvalues are monotone increasing. 

\begin{theorem}\label{benignithm}
Let us choose an $k=k_n$ such that $\gamma_{k,t}\in [\lambda_0-\frac{1}2r,\lambda_0+\frac{1}2r]$, and a $q=q_n\in \mathbb{C}^n$ such that $\|q\|_2=1$. Then
\[{\frac{\sqrt{n}}{\sigma_t(q,k)}} \langle q,u_k\rangle \]
converge in distribution to a standard complex Gaussian random variable. 

More generally, let $m$ be a positive integer, and let  $k_1<k_2<\dots<k_m$ be such that $\gamma_{k_i,t}\in [\lambda_0-\frac{1}2r,\lambda_0+\frac{1}2r]$ for all $i$. Then
\[\frac{\sqrt{n}}{\sigma_t(q,k)} \langle q,u_{k_j}\rangle, \qquad j=1,\ldots ,m \]
converge in distribution to $m$ independent standard complex Gaussian  random variables.

\end{theorem}
\begin{proof}
This is a special case of \cite[Theorem 1.3]{benigni2020eigenvectors}, except one small detail. In that paper, it is assumed that $q\in\mathbb{R}^n$. However, it is easy to extend this statement to the case $q\in \mathbb{C}^n\backslash \mathbb{R}^n$. Indeed, one can find a diagonal unitary matrix $R=R_n$ such that $q'=R q\in \mathbb{R}^n$. Let $u_k'=R q$, then $u_k'$ is an eigenvector of  $R(D+\sqrt{t}W)R^*$ corresponding to the $k$th eigenvalue with a uniform random phase. Since $R(D+\sqrt{t}W)R^*$ and $D+\sqrt{t}W$ have the same distribution, $u_k$ and $u_k'$ has the same distribution. Note that $\langle q,u_k\rangle=\langle R^{*} q',u_k\rangle=\langle q',R u_k\rangle=\langle  q',u_k'\rangle $. But  $\langle  q',u_k'\rangle$ has the same distribution as $\langle q', u_k\rangle$. It is also easy to see that $\sigma_{t}^2(q,k)=\sigma_{t}^2(q',k)$.  So the statement follows. 
\end{proof}

See also \cite{marcinek2020high}, where a similar theorem is proved for non-diagonal initial matrix.

\subsection{A version of Theorem \ref{lemmaquenched1A} where the regularity holds in every point}\label{secallpoint}

\begin{theorem}[Eigenvector limits of GUE-perturbed locally convergent matrices]\label{lemmaquenched1}
Let $A=A_n$ be an $n\times n$ Hermitian matrices,  $E=E_n\in \mathbb R$,  $0<\varepsilon<\frac{1}3$ and $t=t_n\in [n^{\varepsilon-1},n^{-2\varepsilon}]$. Let 
$m(z):=\Tr((A-z)^{-1})/n$, the Stieltjes transform of the empirical measure of eigenvalues.
Assume that there is $c>0$ so that $\|A\|\le c$ and 
\begin{equation}\label{condreg}1/c\le \Im m(\lambda+\eta i)\le c, \qquad 
\mbox{ for all }n, (\lambda,\eta)\in [E\pm t n^{\varepsilon}]\times[t n^{-\varepsilon},10].
\end{equation}

Let $\eta_\ell=\eta_{\ell,n}$ and $\eta_u=\eta_{u,n}$ be so that $0<\eta_\ell<\eta_u$ and for all large enough $n$  
and all $ \lambda\in [E\pm t n^{\varepsilon}]$, we have
\begin{equation}\label{cond2}\Im m(\lambda+i\eta_{\ell})>\frac{\eta_{\ell}}t \qquad \text{ and } \qquad \Im m(\lambda+i\eta_{u})<\frac{\eta_{u}}t.
\end{equation}
Assume that there is an $\mathbb N\times \mathbb N$ matrix $M$ so that for $\eta=\eta_{u}$ and $\eta=\eta_{\ell}$  as $n\to\infty$,
\begin{align}\label{limitM} \sup_{\lambda\in [E\pm t n^{\varepsilon}]} \left|\left(\frac{t}{(A-\lambda)^2+\eta^2}\right)(x,y)-M(x,y)\right|\to 0 \qquad \mbox{for all }x,y\in \mathbb{N}.
\end{align}
Let $W=W_n$ be an $n\times n$ GUE matrix with  entries of variance $n^{-1}$.
Let 
\[\lambda_{1}\le \dots\le\lambda_{n}\quad \mbox{
be the eigenvalues of }\quad
A+\sqrt{t} W,\]
and let $u_j$ be the corresponding eigenvectors with $\ell^2$ norm $\sqrt{n}$, and  phases  chosen independently uniformly at random. Appending with zeros makes $u_j$ a vector in $\mathbb{C}^\mathbb{N}$.


Let $\mathcal{E}\subset [E\pm t n^{\varepsilon}/2]$  be an interval of  length $|\mathcal{E}|\ge n^{\varepsilon-1}$. 

Then ${\mathbf E}|\{j\,:\,\lambda_{j}\in \mathcal{E}\}|\ge n|\mathcal E|\frac{\eta_{\ell}}{\pi t}-n^{o(1)}$, and 
the random empirical measure of eigenvectors 
\[\frac{1}{|\{j\,:\,\lambda_{j}\in \mathcal{E}\}|} \sum_{j:\lambda_j\in \mathcal E} \delta_{u_j}\]
converges in probability to the law of the complex Gaussian process on $\mathbb N$ with covariance $M$.
\end{theorem}

Now we prove Theorem~\ref{lemmaquenched1}. Let $\eta_*=\eta_{*,n}=n^{-\varepsilon}t_n$, $r=r_n=n^\varepsilon t_n$ and let $Z$ be the complex Gaussian process on $\mathbb N$ with covariance $M$.

Let us consider the empirical measure of the eigenvalues of $A$, then take its free convolution with the semicircle law of variance $t$. Let $\gamma_{k,t}$ be quantiles of this free convolution, defined as in  \eqref{quantiledef}.
\begin{lemma}\label{oneeigenvector}
Let us choose $k=k_n$ such that $\gamma_{k,t}\in [E-\frac{1}2 r,E+\frac{1}2 r]$. Then ${u}_k$ converges to $Z$.
\end{lemma}

\begin{proof}
As before, let
\[v_t(\lambda)=\inf \left \{v\ge 0 \,:\,\frac{m(\lambda+vi)}{v}\le \frac{1}t\right\}.\]
Since the map $v\mapsto \frac{m(\lambda+vi)}{v}$ is monotone decreasing, it follows from condition \eqref{cond2} that \begin{equation}\label{vtin}v_t(\lambda)\in [\eta_\ell,\eta_u]\qquad\text{ for all }\lambda\in [E-r,E+r].\end{equation}  

Let $m_t$ be the Stieltjes transform of the free convolution of the emprical measure on the eigenvalues of $A$ and the semicircle law of variance $t$. Using condition \eqref{condreg} and the assumption that $\|A\|\le c$, it follows that there is a $C_2$ such that \begin{equation}\label{logbound}|m_t(\lambda)|\le C_2 \log(n)\qquad\text{ for all }\lambda\in [E-\frac{1}{2}r,E+\frac{1}{2}r].\end{equation} See the proof of  \cite[Lemma 7.2]{landon2017convergence}. As before, we define
\[\psi_t(\lambda)=\lambda-t\int\frac{x-\lambda}{(x-\lambda)^2+v_t(\lambda)^2}\mu(dx),\]
where $\mu$ is the empirical measure on the eigenvalues of $A$. 

From now on assume that $n$ is large enough. From \eqref{reimfc1} and \eqref{logbound}, we see that 
\[|\psi_t^{-1}(\lambda)-\lambda|\le t C_2 \log n<\frac{1}2 r\qquad \text{ for all }\lambda\in [E-\frac{1}{2}r,E+\frac{1}{2}r].\]
Therefore, 
\begin{equation}\label{psiinpmr}\psi_t^{-1}(\lambda)\in [E-r,E+r]\qquad\text{ for all }\lambda\in [E-\frac{1}{2}r,E+\frac{1}{2}r] .\end{equation} 
Combining it with \eqref{vtin}, we see that
\begin{equation}\label{vtinwellwu}v_t(\psi^{-1}(\lambda))\in [\eta_\ell,\eta_u]\qquad\text{ for all }\lambda\in [E-\frac{1}{2}r,E+\frac{1}{2}r].\end{equation}

Let $U$ be a unitary matrix which diagonalizes $A$, that is, $D=UAU^*$ is a diagonal matrix.

Let $q\in \mathbb{C}^\mathbb{N}$ be a vector with finite support such that $\|q\|_2=1$. Let $q_n\in \mathbb{C}^{\{1,2,\dots,n\}}$ be obtained from $q$ by truncating it.

Let $\hat{q}=U q_n$. Recalling the definition of $\sigma^2_t$ from \eqref{sigmadef}, we see that
\begin{align}\sigma^2_t (\hat{q},k)&=\left\langle \hat{q}, \frac{t}{(D-\psi^{-1}(\gamma_{k,t}))^2+v_t(\psi^{-1}(\gamma_{k,t}))^2} \hat{q}\right\rangle\nonumber\\&=\left\langle q_n, \frac{t}{(A-\psi^{-1}(\gamma_{k,t}))^2+v_t(\psi^{-1}(\gamma_{k,t}))^2} q_n\right\rangle.\label{sigmaAn}
\end{align}

\begin{lemma}\label{lemmasigmalimit}
For any $x,y\in \mathbb{N}$, we have
\[\lim_{n\to\infty}\left(\frac{t}{(A-\psi^{-1}(\gamma_{k,t}))^2+v_t(\psi^{-1}(\gamma_{k,t}))^2}\right)(x,y)=M(x,y).\]
\end{lemma}
\begin{proof}
Let $\lambda=\psi^{-1}(\gamma_{k,t})$ and $\eta=v_t(\psi^{-1}(\gamma_{k,t}))$. We have seen in \eqref{psiinpmr} and \eqref{vtinwellwu} that
\[\lambda\in [E-r,E+r]\text{ and }\eta\in [\eta_\ell,\eta_u].\]

Let us introduce the notation
\[B(\eta,x,y)=\frac{t}{(A-\lambda)^2+\eta^2}(x,y).\]

Since $0\le \frac{t}{u^2+\eta_\ell^2}-\frac{t}{u^2+\eta^2} \le \frac{t}{u^2+\eta_\ell^2}-\frac{t}{u^2+\eta_u^2}$ for all $u\in \mathbb{R}$, using Lemma~\ref{posdef}, we see that
\begin{align*}\left|B(\eta_\ell,x,x)-B(\eta,x,x)\right|&\le \left|B(\eta_\ell,x,x)-B(\eta_u,x,x)\right|\\&\le 
\left|B(\eta_\ell,x,x)-M(x,x)\right|+\left|B(\eta_u,x,x)-M(x,x)\right|.
\end{align*}

Using condition \eqref{limitM} both terms on the right hand side converge to zero which gives us
\[\lim_{n\to\infty}\left|B(\eta_\ell,x,x)-B(\eta,x,x)\right|=0,\]
and the same is true with $y$ in place of $x$.

Using Lemma \ref{posdef} again, we see that
\begin{equation}\label{secondterm}\lim_{n\to\infty}\left|B(\eta_\ell,x,y)-B(\eta,x,y)\right|
\le \lim_{n\to\infty}\sqrt{\left|B(\eta_\ell,x,x)-B(\eta,x,x)\right|\left|B(\eta_\ell,y,y)-B(\eta,y,y)\right|}=0
\end{equation}

Clearly,
\[
\left|B(\eta,x,y)-M(x,y)\right|\le \left|B(\eta_\ell,x,y)-M(x,y)\right|+\left|B(\eta_\ell,x,y)-B(\eta,x,y)\right|, 
\]
where both term on the right hand side converge to zero, the first one due to condition \eqref{limitM}, the second one due to \eqref{secondterm}. 
\end{proof}

Combining Lemma \ref{lemmasigmalimit} with \eqref{sigmaAn}, we see that
\begin{equation*}
\lim_{n\to\infty} \sigma_t^2(\hat{q},k)= \langle  q,Mq\rangle.
\end{equation*}

Note that $\hat{u}_k=U u_k$ is the $k$th eigenvector of $D+\sqrt{t} UWU^*$ with a uniform random phase and length $\sqrt{n}$. Note that $D+\sqrt{t}UWU^*$ has the same distribution as $D+\sqrt{t}W$ as it follows from Lemma \ref{invariant}. Thus, applying Theorem \ref{benignithm}, we see that
$\langle \hat{q},\hat{u}_k\rangle/\sqrt{\langle  q,Mq\rangle}$
converge in distribution to a standard complex Gaussian random variable. 
Since $\langle \hat{q},\hat{u}_k\rangle=\langle Uq_n,U u_k\rangle=\langle {q}_n,{u}_k\rangle$, we see that $\langle {q}_n,{u}_k\rangle/\sqrt{\langle  q,Mq\rangle} $
converge in distribution to a standard complex Gaussian random variable. 

Thus, the statement follows from Cram\'er-Wold theorem.
\end{proof}

Let us choose $\mathcal{I}=\mathcal{I}_n$ such that 
\[\mathcal{I}\subset \{k\,:\, \gamma_{k,t}\in [E-\frac{1}2 r,E+\frac{1}2 r]\}.\]

\begin{lemma}\label{lemmaquenched0}
Assume that $\lim_{n\to\infty} |\mathcal{I}|=\infty$. Let $f:\mathbb{C}^\mathbb{N}\to \mathbb{R}$ be a bounded continuous function. Then
\[\frac{1}{|\mathcal{I}|}\sum_{k\in \mathcal{I}} f({u}_k)\to \mathbb{E}f(Z).\]

\end{lemma}
\begin{proof}
From the Cram\'er-Wold theorem, it is enough to prove the statement for function of the form
\[f(u)=g(\Re \langle q, u\rangle )\]
where $g$ is a bounded continuous function and $q$ is unit vector in $\mathbb{C}^\mathbb{N}$ with bounded support.

Let $N,N_1,N_2,\dots$ be i.i.d. normal random variables with mean zero and variance $\frac{1}2\langle q,Mq \rangle$. Note that $\Re \langle Z, q\rangle$ has the same distribution as $N$. Thus,
$\mathbb{E}f(Z)=\mathbb{E}g(N)$.

Let $\varepsilon>0$. 
From the law of large numbers, we see that for a large enough $m$, we have
\[\mathbb{E}\left|\frac{1}{m}\sum_{i=1}^m {g(N_i)}-\mathbb{E}g(N)\right|<\varepsilon.\]

Let $K=\{k_1,k_2,\dots,k_m\}$ be an $m$ element subset of $\mathcal{I}$. As in the proof of Lemma \ref{oneeigenvector}, one can prove that
$\left(\Re \langle q,{u}_{k_i}\rangle \right)_{i=1}^m$
converge weakly to $(N_1,N_2,\dots, N_m)$.

In particular,
\[\left|\frac{1}{m}\sum_{i=1}^{m}{g(\Re \langle q,{u}_{k_i}\rangle)}-\mathbb{E}g(N)\right|\rightarrow \left|\frac{1}{m}\sum_{i=1}^m {g(N_i)}-\mathbb{E}g(N)\right|\]
weakly. As this is true for any choice of $K$ as above, we see that for any large enough $n$, for all $m$ element subset $K$ of $\mathcal{I}$, we have
\[\mathbb{E}\left|\frac{1}{m}\sum_{i=1}^m {g(\Re \langle q,{u}_{k_i}\rangle)}-\mathbb{E}g(N)\right|<2\varepsilon.\]

Let $\mathcal{I}^{(m)}$ be the set of all $m$ element subsets of $\mathcal{I}$. Then for any large enough $n$, we have
\[\mathbb{E}\left|\frac{1}{|\mathcal{I}|}\sum_{k\in \mathcal{I}}g(\Re \langle q,{u}_{k}\rangle)-\mathbb{E}g(N)\right|\le \frac{1}{|\mathcal{I}^{(m)}|}\sum_{\{k_1,\dots,k_m\}\in \mathcal{I}^{(m)}}\mathbb{E}\left|\frac{1}{m}\sum_{i=1}^m {g(\Re \langle q,{u}_{k_i}\rangle)}-\mathbb{E}g(N)\right|\le 2\varepsilon.
\]
\end{proof}
The following simple lemma will be used later. We omit the straightforward proof.  
\begin{lemma}\label{lemmameandif}
Let $a_j$ be real numbers indexed by a finite set $J$.  Then for $I\subset J$,
\[\Big| \tfrac{1}{|I|}\sum_{i\in I}a
_i-\tfrac{1}{|J|}\sum_{j\in J}a_j \Big|\le 2(1 -\tfrac{|I|}{|J|})\max |a_i|.\]
\end{lemma}

We need the following rigidity result. 
\begin{lemma}(\cite[Theorem 3.5]{landon2017convergence})\label{rigidity}
Let $\kappa>0$ and $L>0$. For any large enough $n$, for any $\gamma_{i,t}\in [E-\frac{3}4r,E+\frac{3}4r]$, we have
\[\mathbb{P}\left(|\lambda_{i,t}-\gamma_{i,t}|\ge \frac{n^\kappa}{n}\right)<n^{-L}.\]
\end{lemma}

We also have the following estimate on the density of free convolution.

\begin{lemma}(\cite[Lemma 7.2]{landon2017convergence})\label{ptbound}
There are constants $0<d<D<\infty$ such that for all large enough $n$, we have
$d\le p_t(\lambda)\le D$
for all $\lambda\in [E-\frac{3}4r,E+\frac{3}4r]$.
\end{lemma}

Now we are ready to prove Theorem~\ref{lemmaquenched1}. 

Let $\mathcal{E}=[E_\ell,E_u]$. Choose $0<\kappa<\varepsilon$.  
Let $J=\{k\,:\,\lambda_{k,t}\in \mathcal{E}\}$,
\[\mathcal{I}=\left\{k \,:\, E_\ell+\frac{n^\kappa}{n}\le \gamma_{k,t} \le E_u -\frac{n^\kappa}{n}\right\}\text{ and } \mathcal{J}=\left\{k\,:\, E_\ell-\frac{n^\kappa}{n}\le \gamma_{k,t} \le E_u +\frac{n^\kappa}{n}\right\}.
\]
Using Lemma \ref{ptbound}, we see that
$|\mathcal{J}|-|\mathcal{I}|\le 4D n^\kappa+2,$
and 
\[|\mathcal{I}|\ge dn\left(|\mathcal{E}|-2\frac{n^\kappa}{n}\right)-1\ge d(n^\varepsilon-2n^\kappa)-1.\]

Then it follows 
$\lim_{n\to\infty} |\mathcal{I}|=\infty,$ 
and
\[\lim_{n\to\infty} \frac{|\mathcal{J}|-|\mathcal{I}|}{|\mathcal{I}|}=0.\]

Choosing $L=1$ in Lemma \ref{rigidity}, we see that 
$\mathbb{P}(\mathcal{I}\subset J\subset \mathcal{J})\ge 1-4n^{-1}$. 
On the event above, we have
\[\frac{|J|-|\mathcal{I}|}{|J|}\le \frac{|\mathcal{J}|-|\mathcal{I}|}{|\mathcal{I}|}=o(1).\]

Using Lemma \ref{lemmameandif}, we see that on this event we have
\[\left|\frac{1}{|J|} \sum_{j\in J} f({u}_j)-\frac{1}{|\mathcal{I}|} \sum_{i\in \mathcal{I}} f({u}_i)\right|\le \|f\|_\infty \frac{2(|J|-|\mathcal{I}|)}{|J|}=o(1).\]
Thus, the convergence part of Theorem \ref{lemmaquenched1} follows from Lemma \ref{lemmaquenched0}.

Finally, we prove our estimate on $\mathbb{E}|J|$. 


Combining \eqref{vtinwellwu} with \eqref{ptdef}, we see that $p_t(\lambda)\ge \frac{\eta_\ell}{\pi t}$
for all $\lambda\in[E-\frac{1}2r,E+\frac{1}2r]$. Combing this with Lemma~\ref{ptbound}, we also see that that $\sup_n \frac{\eta_\ell}{\pi t}<\infty$. Thus,
\[|\mathcal{I}|\ge n\int_{E_\ell+n^{-1+\kappa}}^{E_u-n^{-1+\kappa}} p_t(\lambda)d\lambda-2\ge n(E_u-E_l-2n^{-1+\kappa})\frac{\eta_\ell}{\pi t} -2=n(E_u-E_\ell)\frac{\eta_\ell}{\pi t}-O(n^{\kappa}).\]

Using Lemma~\ref{rigidity}, we see that $\mathcal{I}\subset J$ with probability at least $1-n^{-1}$. Thus,
\[\mathbb{E}|J|=(1-n^{-1})|\mathcal{I}|\ge (1-n^{-1})\left(n(E_u-E_\ell)\frac{\eta_\ell}{\pi t}-O(n^{\kappa})\right)\ge n(E_u-E_\ell)\frac{\eta_\ell}{\pi t}-O(n^{\kappa}).\]

\subsection{Continuity of the Stieltjes-transform}
The proof of the next lemma can be found in the appendix.
\begin{lemma}\label{stieltjescont1G} Let $\nu$ be a complex-valued measure on $\mathbb R$, let $\eta>0$
and consider 
$$
f(\lambda)=\int \frac\eta{(t-\lambda)^2+\eta^2} \nu(dt), \qquad 
g(\lambda)=\int \frac\eta{(t-\lambda)^2+\eta^2} |\nu|(dt).
$$
 Then
$
|f'|\le g/\eta
$
and for any interval $J$ of length $2r$, $b\ge 0$ and complex number $z$,  we have
\[\int_J \mathbbm{1}\left(|f(\lambda)-z|\ge b \right)d\lambda\ge \min\left(\eta\frac{\max_J |f-z|-b}{\max_J g},{r}\right).\]
Moreover, if $\nu$ is a probability measure and  $b>0$ is real, we have
\begin{align*}
    \int_J \mathbbm{1}\left(f(\lambda) \ge b\right)d\lambda&\ge \min\left(\eta\frac{\max_J f-b}{\max_I g},{r}\right),\\
\int_J \mathbbm{1}\left(f(\lambda) \le b\right)d\lambda&\ge \min\left(\eta\frac{b-\min_I f}{b},{r}\right).\end{align*}

\end{lemma}

\subsection{The proof of Theorem \ref{lemmaquenched1A}}
In this section, we prove Theorem \ref{lemmaquenched1A}.

Let $\mu_{n,x,y}=\mu_{x,y}$ be the unique complex valued measure such that for all continuous function $f:\mathbb{R}\to\mathbb{R}$, we have
\[f(A)(x,y)=\int f(\lambda) \mu_{x,y}(d\lambda).\]
Note that $\mu_{x,x}$ is always a probability measure. Moreover,
\begin{equation}\label{absolute}
   |\mu_{x,y}|\le \frac{\mu_{x,x}+\mu_{y,y}}2. 
\end{equation}

Let $r=r_n=t_nn^{\varepsilon}$ and $\eta_{*}=\eta_{*,n}=tn^{-\varepsilon}$.

Let us subdivide the interval $[E-\ell+\frac{r}2,E+\ell-\frac{r}2]$ into $2\frac{\ell}{r}-1$ intervals of length $r$. We assume that $\frac{\ell}{r}$ is an integer, to make things more convenient. Obviously, with some care one can handle the case when this is not integer.  Let $\mathcal{I}$ be a the set of these intervals. 

Given an $I=\left[a\pm \frac{r}2\right]$, let $I^F=\left[a\pm {r}\right]$.

Let \[H=\{\eta_*,2\eta_*,\dots, 2^b \eta_*\},\]
where $b$ is the smallest integer such that $2^b \eta_*\ge 10$.  
Let us introduce the notation
\[B(\eta,\lambda,x,y)=\frac{t}{(A-\lambda)^2+\eta^2}(x,y).\]

\begin{lemma}\label{convergeramdomI}
Let $I$ be a uniform random element of $\mathcal{I}$. Fix $x,y\in \mathbb{N}$. Then for  $\eta=\eta_\ell$ or $\eta=\eta_u$, we have that the random variables
\[\sup_{\lambda\in I^F}\left|B(\eta,\lambda,x,y)-M(x,y)\right|\]
converge to zero in probability.

Moreover,
\begin{align*}\lim_{n\to\infty} \mathbb{P}_I\left(\Im m(\lambda+i\eta_\ell)>\frac{\eta_\ell}t\text{ and } \Im m(\lambda+i\eta_u)<\frac{\eta_u}t\text{ for all }\lambda\in I^F\right)&=1,\text {and}\\
\lim_{n\to\infty} \mathbb{P}_I\left(1/(2c)\le \Im m(\lambda+i\eta)\le 2c\text{ for all }\lambda\in I^F, \eta\in H \right)&=1,
\end{align*}
where $\mathbb{P}_I$ denotes the probability over the uniform random choice of $I\in\mathcal{I}$.
\end{lemma}
\begin{proof}
The statement can be proved by combining the assumptions of Theorem~\ref{lemmaquenched1A} with Lemma~\ref{stieltjescont1G}. See the appendix for details.
\end{proof}
\begin{remark}\label{remarklemmaquenched1A}
Note that in the proof of Lemma~\ref{convergeramdomI}, we only used a weaker form of condition~\eqref{condregA} of Theorem~\ref{lemmaquenched1A}. Thus, it turns out that  condition \eqref{condregA} in Theorem~\ref{lemmaquenched1A} can be replaced with a weaker one. Namely, instead of taking the supremum over $[tn^{-\varepsilon},20]$, it is enough to take the supremum over the set $H$.

\end{remark}

It follows from Lemma \ref{convergeramdomI}, that we can find $\mathcal{I}_G\subset \mathcal{I}$, with the property that
\begin{align}\label{legtobbexc}
    \lim_{n\to\infty} \frac{|\mathcal{I}_G|}{|\mathcal{I}|}&=1,\\
\lim_{n\to\infty}\sup_{I\in \mathcal{I}_G}\sup_{\lambda\in I^F}\left|B(\eta,\lambda,x,y)-M(x,y)\right|&=0\qquad\text{ for all }x,y\in \mathbb{N}\text{ and }\eta=\eta_u\text{ or }\eta=\eta_\ell.\nonumber
\end{align}
Moreover, for any $I\in \mathcal{I}_G$, we have
\begin{align*}\Im m(\lambda+i\eta_\ell)>\frac{\eta_\ell}t, \Im m(\lambda+i\eta_u)<\frac{\eta_u}t\text { and }
1/(2c)\le \Im m(\lambda+i\eta)\le 2c\text{ for all }\lambda\in I^F, \eta\in H .
\end{align*}

\begin{lemma}\label{mindenetajo}
If $I\in \mathcal{I}_G$, then 
$1/(4c)<\Im m(\lambda+i\eta)<4c\text{ for all }(\lambda,\eta)\in I^F\times [\eta_*,10].$
\end{lemma}
\begin{proof}
Let $(\lambda,\eta)\in I^F\times [\eta_*,10]$. Choose a non-negative integer $j$ such that $\eta_* 2^j\le \eta\le \eta_* 2^{j+1}$. Since $I\in \mathcal{I}_G$, we have
\[1/(2c)<\Im m(\lambda+i\eta_*2^j)<2c\text{ and }1/(2c)<\Im m(\lambda+i\eta_*2^{j+1})<2c.\]
Note that the map $\eta\mapsto \eta\cdot \Im m(\lambda+i\eta)$ is increasing on $[0,\infty)$. In particular,
\[\eta_*2^j \Im m(\lambda+i\eta_*2^j)\le \eta \Im m(\lambda+ i\eta)\le \eta_*2^{j+1} \Im m(\lambda+i\eta_*2^{j+1}),\]
the statement follows. 
\end{proof}

Let $\lambda_{1,0}\le\lambda_{2,0}\le \cdots \le \lambda_{n,0}$ be the eigenvalues of $A$, and  $\lambda_{1,t}\le\lambda_{2,t}\le \cdots \le \lambda_{n,t}$ be the eigenvalues of $A+\sqrt{t} W$.

Applying Theorem~\ref{lemmaquenched1}, we get the following lemma.
\begin{lemma}\label{lemmaquenchedcor}
For each $n$, let us choose $I=I_n\in \mathcal{I}_G$, then
$\mathbb{E} |\{j\,:\,\lambda_{j,t}\in I\}|\ge (1-o(1))r n \frac{\eta_{\ell}}{t}$, and
\[\frac{1}{|\{j\,:\,\lambda_{j,t}\in I\}|}\sum_{j:\lambda_{j,t}\in I} \delta_{u_j}\]
converges in probability to the law of the complex Gaussian process on $\mathbb{N}$ with covariance $M$.

\end{lemma}

\begin{lemma}\label{lemmaoribecs}
Let $I\in \mathcal{I}$, then  
\[|\{j\,:\,\lambda_{j,0}\in I\}|\le O(n) \int_{I} m(\lambda+i\eta_{u})d\lambda.\]

Furthermore, if we assume that $I\in \mathcal{I}_G$, then
\[|\{j\,:\,\lambda_{j,0}\in I\}|\le (1+o(1))r n\frac{\eta_{u}}{t\pi}. \]
\end{lemma}
\begin{proof}
Let $I=[a,b]$. Let use define $L=a-t n^{\varepsilon/2}$ and $U=b+t n^{\varepsilon/2}$. Observe that if $\lambda_{j,0}\in I$, then using the fact that $\eta_{u}=o(t n^{\varepsilon/2})$, we have
\[\int_{L}^U \frac{\eta_{u}}{(\lambda_{j,0}-\lambda)^2+\eta_{u}^2}d\lambda\ge \pi(1-o(1)) \text{ and }\int_{a}^b \frac{\eta_{u}}{(\lambda_{j,0}-\lambda)^2+\eta_{u}^2}d\lambda\ge \frac{\pi(1-o(1))}2.\]
Therefore
\begin{align}\label{eq1274635}\int_L^U m(\lambda+i\eta_{u})&\ge \pi(1-o(1)) \frac{|\{j\,:\,\lambda_{j,0}\in I\}|}n,\text{ and } \\\int_a^b m(\lambda+i\eta_{u})&\ge \frac{\pi(1-o(1))}2 \frac{|\{j\,:\,\lambda_{j,0}\in I\}|}n.\nonumber
\end{align}
Thus, the first statement follows. 

Now we prove the second statement. Using the fact that $I\in\mathcal{I}_G$,
\[\int_L^U m(\lambda+i\eta_{u})\le (U-L)\frac{\eta_{u}}{t}=(1+o(1))r\frac{\eta_{u}}{t}.\]
Combining this with \eqref{eq1274635}, the statement follows.
\end{proof}

Let $R=R_n=\frac{\ell}{2}+3\sqrt{t}$. Let $\mathcal{I}_B$ be the set of all intervals in $\mathcal{I}\setminus \mathcal{I}_G$, that intersect $[E\pm R]$. Let $D$ be the union of the intervals in $\mathcal{I}_B$. 
Since $R+r<\frac{\ell}2$ for all large enough $n$, we see that $D\subset [E\pm \ell]$.

Using \eqref{legtobbexc}, we see that $|\mathcal{I}_B|=o(|\mathcal{I}|)$. Thus,
\[\mathbb{P}_\lambda(\lambda\in D)\le \frac{1}{2\ell}|\mathcal{I}_B|r=o(1).\]

Using Lemma~\ref{lemmaoribecs}, and then the uniform integrability condition of Theorem~\ref{lemmaquenched1A}, we see that
\begin{equation}\label{oribecs1}
    |\{j\,:\,\lambda_{j,0}\in D\}|\le O(n)\int_{D} m(\lambda+i\eta_{u})d\lambda=o(\ell n).
\end{equation}

Consider all the  intervals in $\mathcal{I}_G$ that intersect $[E\pm R]$. There are at most $(1+o(1))\frac{\ell}{r}$ such intervals. Let $K$ be the union of these intervals. Then using Lemma~\ref{lemmaoribecs}, 
\begin{equation}\label{oribecs2}
   |\{j\,:\,\lambda_{j,0}\in K\}|\le (1+o(1))\frac{\ell}{r} (1+o(1))rn\frac{\eta_{u}}{t\pi}\le (1+o(1))\ell n \frac{\eta_{u}}{t\pi}.
\end{equation}

Combining \eqref{oribecs1} and \eqref{oribecs2}, we see that
\begin{equation}\label{oribecs3}
     |\{j\,:\,\lambda_{j,0}\in [E\pm R]\}|\le|\{j\,:\,\lambda_{j,0}\in D\}|+ |\{j\,:\,\lambda_{j,0}\in K\}|= (1+o(1))\ell n \frac{\eta_{u}}{t\pi}.
\end{equation}

Let $\mathcal{I}_E$ the set of intervals in $\mathcal{I}_G$ contained in $[E\pm \ell/2]$. Let $Q$ be the union of these intervals. Let 
\[I=\left\{k\,:\, \lambda_{k,t}\in Q\right\}\text{ and }J=\left\{k\,:\, \lambda_{k,t}\in [E\pm \ell/2]\right\}.\]

Using \eqref{legtobbexc} , we see that $|\mathcal{I}_E|=(1-o(1))\frac{\ell}{r}$. Combining this with Lemma~\ref{lemmaquenchedcor}, we see that 
\[\mathbb{E} |I|\ge |\mathcal{I}_E|(1-o(1))nr\frac{\eta_{\ell}}{\pi t}=(1-o(1))\ell n \frac{\eta_{\ell}}{\pi t}.
\]

Since $|I|\le n$, using Lemma~\ref{normbecs}, we see that
\begin{equation*}
\mathbb{E}\mathbbm{1}(\|W\|\le 3)|I|\ge \mathbb{E}|I|-n\mathbb{P}(\|W\|>3)\ge (1-o(1))\ell n \frac{\eta_{\ell}}{\pi t}.
\end{equation*}

Using Lemma~\ref{weyl}, on the event $\|W\|\le 3$, we have $J\subset\left\{j\,:\,\lambda_{j,0}\in [E\pm R]\right\}.$ 

Thus, on the event $\|W\|\le 3$, we have $|J|\le s$, where $s=s_n=\left|\left\{j\,:\,\lambda_{j,0}\in [E\pm R]\right\}\right|.$

Combining \eqref{oribecs3} and the condition that $\lim_{n\to\infty} \frac{\eta_{u}}{\eta_{\ell}}=1$, we see that 
\[s\le (1+o(1))\ell n \frac{\eta_{u}}{\pi t}\le (1+o(1))\ell n \frac{\eta_{\ell}}{\pi t}.\]

Also
\[s\ge \mathbb{E} \mathbbm{1}(\|W\|\le 3)|I|\ge (1-o(1))\ell n \frac{\eta_{\ell}}{\pi t}.\]

It follows that
\[\mathbb{E}\mathbbm{1}(\|W\|\le 3)(s-|I|)=o(\ell n).\]

Note that $s-|I|\ge 0$ on the event $\|W\|\le 3$. Thus, using Markov's inequality, for every $\varepsilon>0$ we have
\[\lim_{n\to\infty}\mathbb{P}(\|W\|\le 3\text{ and }s-|I|\le \varepsilon \ell n)=1.\]

On the event $\|W\|\le 3$ and $s-|I|\le \varepsilon \ell n$, we have $s-\varepsilon\ell n\le|I|\le |J|\le s$, thus,
\[\frac{|J|-|I|}{|J|}\le \frac{\varepsilon\ell n}{s-\varepsilon\ell n}\le c'\varepsilon\]
for some constant $c'$ not depending on $\varepsilon$. Thus,
 $\frac{|J|-|I|}{|J|}$ converge to zero in probability.  

Consider any bounded continuous function $f:\mathbb{C}^\mathbb{N}\to\mathbb{R}$. Using Lemma~\ref{lemmameandif}, it follows that
\[\frac{1}{|J|}\sum_{j\in J} f(u_j)-\frac{1}{|I|}\sum_{i\in I} f(u_i)\]
converge to zero in probability.
Therefore, it is enough to prove that $\frac{1}{|I|}\sum_{i\in I} f(u_i)\to \mathbb{E}f(Z)$. This statement follows from Lemma~\ref{lemmaquenchedcor} .

\section{Large enough perturbations of the torus}

In this section, we explain how to apply Theorem~\ref{lemmaquenched1A} to obtain the convergence part of Theorem~\ref{t:main}. In Theorem~\ref{t:main}, we consider the matrix $A+n^{-\gamma}W$. If one wants to use Theorem~\ref{lemmaquenched1A}, this corresponds to the choice of $t=\left(n^2\right)^{-\gamma}=n^{-2\gamma}$.  Note that in Theorem~\ref{lemmaquenched1A} the matrices are assumed to be indexed with the first few positive integers. However, in our setting here,  $A$ is indexed with $\{1,2,\dots,n\}^2$ and $M$ will be indexed with $\mathbb{Z}^2$. So to apply Theorem~\ref{lemmaquenched1A}, we need to identify $\mathbb{Z}^2$ with $\mathbb{N}$, and  $\{1,2,\dots,n\}^2$ with $\{1,2,\dots,n^2\}$. We do this by choosing  bijections $i\mapsto (x_i,y_i)$ from $\mathbb{N}$ to $\mathbb{Z}^2$ and  $i\mapsto (x_i^{(n)},y_i^{(n)})$ from $\{1,2,\dots,n^2\}$ to $\{1,2,\dots,n\}^2$. Since we are interested in the local behaviour of the eigenvector in the neighborhood of $o_n$, these bijections must be compatible in the sense that for all $i\in \mathbb{N}$, we have $(x_i^{(n)},y_i^{(n)})=o_n+(x_i,y_i)$ for all large enough $n$. 

Choose an $\varepsilon>0$ such that
\[15\varepsilon<\min(2(1-\gamma-\varepsilon),1/6) \text{ and } 8\varepsilon<\gamma.\]

We set
\[
    \delta=3\varepsilon,\quad \ell=2n^{-\delta},\quad \eta_\ell=(\pi \varrho_{0,0}(E)-n^{-\varepsilon})t,\quad
    \eta_u=(\pi \varrho_{0,0}(E)+n^{-\varepsilon})t.
\]

Let $M$ be an $\mathbb{N}^2\times \mathbb{N}^2$ matrix defined as
\[M((x+\ux,y+v),(x,\uy))=\frac{\varrho_{u,v}(E)}{\varrho_{0,0}(E)},\]
where $\varrho_{\ux,\uy}$ is the density of the spectral measure of the adjacency operator of $\mathbb{Z}^2$, see Section~\ref{subsecspec}.

We say that $({\rm{c}},{\rm{d}})$ is nice for a given $n$, if for all
\begin{multline*}\eta\in \{\eta_\ell, \eta_u\}\cup \{t n^{-2\varepsilon},2t n^{-2\varepsilon},2^2t n^{-2\varepsilon},\dots,2^bt n^{-2\varepsilon}\},\quad (\ux,\uy)\in \{(\ux,\uy)\in \mathbb{Z}^2\,:\, |\ux|+|\uy|\le n^\varepsilon\},
\end{multline*}
we have
\[\int_{-\infty}^\infty \left|d_\ux*d_\uy*\kappa_\eta(\lambda)-\mu_{{\rm{c}},\ux}*\mu_{{\rm{d}},\uy}*\kappa_\eta(\lambda)\right|^2\le 2\ell n^{-9\varepsilon} . \]

(Here as before $b$ is the smallest integer such that $2^b t n^{-2\varepsilon}\ge 10$. Note that we have $2\varepsilon$ in the exponent in place of $\varepsilon$, because we have $n^2\times n^2$ matrices. Recall that $d_u$ is the density of the spectral measure of the adjacency operator of $\mathbb{Z}$, see Section~\ref{subsecspec}. Furthermore, $\kappa_\eta$ is  $\pi$ times the  Cauchy distribution, see \eqref{kappaetadef}.)

We need the following corollary of Theorem \ref{torusthm1}.
\begin{corollary}\label{regcorollary}
Let $2>\varepsilon_1>0$. Then for all $20\ge \eta\ge n^{-2+\varepsilon_1}$ and $|\ux|+|\uy|\le n^{1/6}$, we have that
\[\int_{-\infty}^\infty \Var\left(\int \frac\eta{(x-\lambda)^2+\eta^2} \mu_{C,D,\ux,\uy}(dx)\right)d\lambda= O\left(n^{-\min(\varepsilon_1,1/6)}\right).\]
 
\end{corollary}
\begin{proof}
We can apply Theorem \ref{torusthm1} with the choice of $\varepsilon=1/12.$
\end{proof}

Combining Corollary~\ref{regcorollary} with Markov's inequality and the union bound, we see that if we choose $C,D$ uniformly from $[0,1]^2$, then
\[\mathbb{P}\left((C,D)\text{ is not nice}\right)\le O(n^{2\varepsilon}\log(n)\ell^{-1} n^{9\varepsilon}n^{-\min(2(1-\gamma-\varepsilon),1/6)})=o(1)\]
where the last equality follows from our choice of parameters. Thus, to obtain the convergence part of Theorem~\ref{t:main}, it is enough to prove the following statement, which is essentially the same as Theorem~\ref{t:main}, but  the random matrix  $A_{C,D}$ is replaced with the deterministic  matrix $A_{n,{\rm{c}},{\rm{d}}}$. 
\begin{proposition}\label{prop47}
Let $\gamma$ and $\delta$ be chosen as above. For each $n$, let us fix a nice pair $({\rm{c}},{\rm{d}})=({\rm{c}}_n,{\rm{d}}_n)$.  Consider eigenvalues of 
$$
A_{{\rm{c}},{\rm{d}}}+n^{-\gamma}W,
$$
(the adjacency matrix of the discrete torus with boundary conditions given by ${\rm{c}}$ and ${\rm{d}}$ plus an independent scaled GUE) in the interval  $(E\pm n^{-\delta})$. Let $u_1,u_2,\dots,u_m$ be the corresponding eigenfunctions with independent uniform random phases and $\ell^2$-norm $\sqrt{n}$. Let $o_n\in \{1,\ldots,n\}^2$ be a sequence so that for both coordinates $j$ we have $(o_n)_j,n-(o_n)_j\to \infty$. Let $n\to\infty$ and consider the random measures 
\[\frac{1}{m} \sum_{i=1}^m \delta_{u_i(\cdot+o_n)}.\]
 These measures converge in probability to the
law of the complex Gaussian wave $Z_E$, that is, the complex Gaussian process on $\mathbb{Z}^2$ with covariance $M$.
\end{proposition}
\begin{proof}
With the above choices the conditions of Theorem~\ref{lemmaquenched1A} are satisfied for $A=A_{{\rm{c}},{\rm{d}}}$. 
See the appendix for a proof.
\end{proof}

\section{Small perturbations}
\subsection{The GUE resolvent flow}
Let $B_{u,v}(t)$ denote an $n\times n$ array of independent standard complex Brownian motions (the normalization is $E|B_{u,v}(t)|^2=t$). Then
$\tilde W_t = (B+B^*)/(\sqrt{2n})$. Then for $t$ fixed, $\tilde W_t$ is GUE with entries of variance $t/n$. 




Let $D$ be a real diagonal matrix,  let $W_t=D+\tilde W_t$. 
Let $\lambda_{1,t}\le \lambda_{2,t}\le\dots\le\lambda_{n,t}$ be the eigenvalues of ${W}_t$, and let $u_{1,t},u_{2,t},\dots,u_{n,t}$ be a corresponding orthonormal basis of eigenvectors. 
Let 
\[m_t(z)=\frac{1}n\Tr({W}_t-z)^{-1}= \frac{1}{n}\sum_{j=1}^n \frac{1}{\lambda_{j,t}-z}.\]

Moreover, for $x\in \{1,2,\dots,n\}$, we define
\[G_t(x,x,z)=\langle ({W}_t-z)^{-1}\delta_x,  \delta_x\rangle=\sum_{j=1}^n |u_{j,t}(x)|^2 \frac{1}{\lambda_{j,t}-z}.\]

The next proposition describes how  $m_t(z)$ and $G_t(x,x,z)$ evolve over time. Let $R_t(z)=({W}_t-z)^{-1}$. Define the martingales $M_t(x,x,z)$ and $M_t(z)$ by 
\[dM_t(x,x,z)=-\langle R_t(z)(d\tilde{W}_t)R_t(z) \delta_x,\delta_x\rangle\text{ and }M_t(z)=\sum_x M_t(x,x,z)/n.\]
Let $[ \cdot  ]$ denote quadratic variation.

The following are shown in~\cite{von2018phase}, Section 4. The details there are presented for the GOE, but as the authors note, the GUE version is proved the same way. 
\begin{proposition}[\cite{von2018phase}]\label{vonprop}
We have the following properties of the GUE resolvent flow: 
\begin{align*}
dm_t(z)&=m_t(z)\frac{\partial}{\partial z} m_t(z)dt+dM_t(z),
\\
[ M_t(z)]&=\frac{1}{n^3} \int_0^t \sum_{k=1}^n\frac{1}{|\lambda_{k,t}-z|^4}
\le\frac{1}{n^2 (\Im z)^3} \int_0^t \Im m_s(z)ds.
\\
dG_t(x,x,z)&=m_t(z)\frac{\partial}{\partial z} G_t(x,x,z)dt+dM_t(x,x,z),\\
[ M_t(x,x,z) ]&=\frac{1}{n(\Im z)^2}\int_0^t (\Im G_s(x,x,z))^2 ds,
\end{align*}
\end{proposition}
Observe that
\begin{equation}\label{partialzm}
    \left|\frac{\partial}{\partial z} m_t(z)\right|=\left|\frac{1}n\sum_{k=1}^n \frac{1}{(\lambda_{k,t}-z)^2}\right|\le \frac{1}n\sum_{k=1}^n \frac{1}{|\lambda_{k,t}-z|^2}=\frac{\Im m_t(z)}{\Im z},
\end{equation}
and similarly,
\begin{equation}\label{partialzG}
    \left|\frac{\partial}{\partial z} G_t(x,x,z)\right|\le \frac{\Im G_t(x,x,z)}{\Im z}.
\end{equation}

For some recent applications of the resolvent flow, see \cite{von2018phase,adhikari2020dyson, bourgade2021extreme}.

\subsection{Stability}

 Choose a complex number $z$ such that $\Im z=n^{-1}$. In this section, we will see that if $t=o(n)$, then $G_0(x,x,z)$ and $G_t(x,x,z)$ are close to each other provided that $m_0(z)$ is small.

\begin{lemma}\label{taubecs}
Assume that $\Im z=n^{-1}$, $|m_0(z)|\le (nt)^{-1/2}$.
Let \[\tau=\inf \{s\ge 0\,:\, |m_s(z)|\ge 2(nt)^{-1/2}\}.\]
Then $\mathbb{P}(\tau\le t)\le  O(\sqrt{tn}).$
\end{lemma}
\begin{proof}
We may assume that $t\le n^{-1}$, because the statement is trivial for $t>n^{-1}$.

By Proposition~\ref{vonprop}, we have
\[m_{t\wedge \tau}(z)-m_0(z)=\int_{0}^{t\wedge \tau} m_s(z)\frac{\partial}{\partial z} m_s(z)ds+M_{t\wedge\tau}(z).\]

For $s<\tau$, we have 
\[\left|m_s(z)\frac{\partial}{\partial z} m_s(z)\right|\le |m_s(z)|\frac{\Im m_s(z)}{\Im z}\le 4t^{-1},\]
where at the first inequality, we used \eqref{partialzm}. Thus,
\[\left|\int_{0}^{t\wedge\tau} m_s(z)\frac{\partial}{\partial z} m_s(z)ds\right|\le 4.\]

By Proposition~\ref{vonprop}, we have
\[[ M_{t\wedge \tau}(z)] \le \frac{1}{n^2(\Im z)^3}\int_{0}^{t\wedge \tau} \Im m_s(z) ds\le  2(nt)^{-1/2}nt\le 2.  \]

Thus, the Burkholder-Davis-Gundy inequality gives us, that \[\mathbb{E} |M_{t\wedge \tau}(z)|\le c_1 \mathbb{E} \sqrt{[ M_{t\wedge \tau}(z)]}\le c_1\sqrt{2}\]
for some universal constant $c_1$. Therefore $\mathbb{E} |m_{t\wedge \tau}(z)-m_0(z)|\le(c_1\sqrt{2}+4)$.
Note that the event $\tau\le t$ is contained in the event $|m_{t\wedge \tau}(z)-m_0(z)|\ge (nt)^{-1/2}$. By Markov's inequality, we have
\[\mathbb{P}(\tau\le t)\le \mathbb{P}(|m_{t\wedge \tau}(z)-m_0(z)|\ge (nt)^{-1/2})\le \frac{\mathbb{E} |m_{t\wedge\tau}(z)-m_0(z)|}{(nt)^{-1/2}} \le (c_1\sqrt{2}+4)\sqrt{nt}.\qedhere\]
\end{proof}

\begin{lemma}\label{stab1}
Let $z=\lambda_{x,0}+in^{-1}$. Assume that $|m_0(z)|\le (nt)^{-1/2}$. Then
\[\mathbb{E}\Big[(n-\Im G_t(x,x,z))\mathbbm{1}(\tau>t)\Big]\le O(n\sqrt{nt}).\]
\end{lemma}
\begin{proof}
By definition $0\le \Im G_s(x,x,z)\le n$ so using Proposition~\ref{vonprop} we have
\[[ M_t(x,x,z)]=\frac{1}{n(\Im z)^2}\int_0^t (\Im G_s(x,x,z))^2 ds\le tn^3.\]
Combining this with Burkholder-Davis-Gundy inequality, we see that $\mathbb{E}|M_t(x,x,z)|\le c_1 n\sqrt{tn}$.
On the event $\tau>t$, 
\[\left|\int_0^t m_s(z)\frac{\partial}{\partial z} G_s(x,x,z)ds\right|\le 2t(nt)^{-1/2}n^2=2n\sqrt{nt},\] 
where the bound on $\tfrac{\partial}{\partial z}G$ follows from \eqref{partialzG}. Therefore,
\begin{align*}
    \mathbb{E}\Big((n-\Im G_t(x,x,z))\mathbbm{1}(\tau>t)\Big)&\le \mathbb{E}\Big(|G_t(x,x,z)-G_0(x,x,z)|\mathbbm{1}(\tau>t)\Big)\\&\le \mathbb{E}\left|\mathbbm{1}(\tau>t)\int_0^t m_s(z)\frac{\partial}{\partial z} G_s(x,x,z)ds\right|+\mathbb{E}|M_t(x,x,z)|\\&\le (2+c_1)n\sqrt{nt}.  &\qquad\qedhere
\end{align*}
\end{proof}

\subsection{Concentration of eigenvectors -- The proof of Theorem~\ref{lemmaconc01}}

\begin{lemma}\label{lemmaconc0}

Let $z=\lambda_{x,0}+in^{-1}$ and assume that  $|m_0(z)|\le (nt)^{-1/2}$. 

Let 
$F=\{k\,:\, |\lambda_{k,t}-\lambda_{x,0}|\ge n^{-1}\}.$ Then
\[\mathbb{E} \sum_{k\in F} |u_{k,t}(x)|^2\le O(\sqrt{nt}). \]
\end{lemma}
\begin{proof}
Note that
\begin{align*}n-\Im G_t(x,x,z)\ge  \sum_{k\in F} |u_{k,t}(x)|^2 \left(n-\frac{n^{-1}}{(\lambda_{k,t}-\lambda_{x,0})^2+n^{-2}}\right)\ge \frac{n}2\sum_{k\in F} |u_{k,t}(x)|^2.
\end{align*}
Combining this with Lemma~\ref{stab1}, we obtain that
\[\mathbb{E}\mathbbm{1}(\tau>t)\sum_{k\in F} |u_{k,t}(x)|^2\le \mathbb{E}\mathbbm{1}(\tau>t) 2n^{-1}(n-\Im G_t(x,x,z))\le O(\sqrt{nt}).\]
Clearly,
\[\mathbb{E}\mathbbm{1}(\tau\le t)\sum_{k\in F} |u_{k,t}(x)|^2\le \mathbb{P}(\tau \le t) \le O(\sqrt{nt})\]
where the second inequality follows from Lemma~\ref{taubecs}.
\end{proof}

Now we prove Theorem~\ref{lemmaconc01}. By Lemma \ref{invariant}, we may assume that $A$ is diagonal. In this case, $v_{k,t}(j)=u_{k,t}(j)$.

Let
\[F_j=\{k\,:\, |\lambda_{k,t}-\lambda_{j,0}|\ge n^{-1}\}\text{ and }J_k=\{j\in K\,:\, |\lambda_{k,t}-\lambda_{j,0}|< n^{-1}\}.\]

 Lemma~\ref{lemmaconc0} gives us
 \[\mathbb{E}\sum_{k=1}^n \sum_{j\in J_k} |u_{k,t}(j)|^2
=\mathbb{E}\sum_{j\in K} \sum_{k\not\in F_j}|u_{k,t}(j)|^2=\mathbb{E}\sum_{j\in K}\left(1-\sum_{k\in F_j} |u_{k,t}(j)|^2\right)\ge |K|(1-O(\sqrt{nt})).\]

Note that $J_k$ is empty unless $k\in K_2$, so 
\begin{equation}\label{conceq111}
 \mathbb{E}\sum_{k\in K_2} \sum_{j\in J_k} |u_{k,t}(j)|^2\ge |K|(1-O(\sqrt{nt})).   
\end{equation}

Let $K_3=\{k\,:\,\lambda_{k,0}\in [E_\ell-3\sqrt{t},E_u+3\sqrt{t}]\}.$ 

On the event $\|\tilde{W}_t\|\le 3\sqrt{t}$, we have $|\lambda_{k,t}-\lambda_{k,0}|\le 3\sqrt{t}$ for all $k$, as it follows from Lemma~\ref{weyl}. Thus, on this event $K_2$ is contained in the set $K_3$, and  

\begin{equation}\label{coneq12}
    \mathbb{E}\frac{1}{|K_2|}\sum_{k\in K_2} \sum_{j\in J_k} |u_{k,t}(j)|^2 \ge \frac{1}{|K_3|}\mathbb{E}\mathbbm{1}(\|\tilde{W}_t\|\le 3\sqrt{t})\sum_{k\in K_2} \sum_{j\in J_k} |u_{k,t}(j)|^2.
\end{equation}

Note that the double sum is bounded by $|K|$, so we have
\begin{align*}\mathbb{E}&\sum_{k\in K_2} \sum_{j\in J_k} |u_{k,t}(j)|^2
\le \mathbb{E}\mathbbm{1}(\|\tilde{W}_t\|
\le  3\sqrt{t})\sum_{k\in K_2} \sum_{j\in J_k} |u_{k,t}(j)|^2 +\mathbb{P}(\|\tilde{W}_t\|> 3\sqrt{t})|K|.
\end{align*}
Using Lemma~\ref{normbecs}, we see that $\mathbb{P}(\|\tilde{W}_t\|> 3\sqrt{t})=O(\sqrt{nt})$. 
Rearranging and combining this with \eqref{conceq111}, we obtain that
\[\mathbb{E}\mathbbm{1}(\|\tilde{W}_t\|\le 3\sqrt{t})\sum_{k\in K_2} \sum_{j\in J_k} |u_{k,t}(j)|^2\ge |K|\left(1-O(\sqrt{nt})
\right).
\]

Inserting this into \eqref{coneq12}, we obtain
\[\mathbb{E}\frac{1}{|K_2|}\sum_{k\in K_2} \sum_{j\in J_k} |u_{k,t}(j)|^2 \ge \frac{|K|}{|K_3|}\left(1-O(\sqrt{nt})
\right).
\qedhere\]

\section{Small perturbations from a random initial condition
}

\subsection{Random initial condition}

Now let $(\lambda_{1,0},\lambda_{2,0},\dots,\lambda_{n,0})$ be random. Let us choose an $E_0$ and an $r=r_n$.

We define
\[\hat{m}_0(z)=\frac{1}{n}\sum_{j=1}^n \frac{1}{|\lambda_{j,0}-z|}.\]
Clearly, $|{m}_0(z)|\le \hat{m}_0(z)$.

For an interval $[a,a+n^{-1}]$ of length $n^{-1}$, let 
\[X_{a}=|\{k\,:\,\lambda_{k,0}\in [a,a+n^{-1}]\}|.\]

 Let us subdivide the interval $[E_0-r+n^{-1},E_0+r-n^{-1}]$ into intervals of length $n^{-1}$. For simplicity, we assume that $r$ is an integer multiple of $n^{-1}$. Let $\mathcal{I}$ be the set of these intervals.

We assume that there is a constant ${g}$ not depending on $n$ such that
\begin{align}\label{conccond1}\mathbb{E} \hat{m}_0(\lambda+in^{-1})&\le {g} \log n&&\text{ for all } \lambda\in [E_0-r,E_0+r],\\
\label{conccond3}\mathbb{E} X_a&\le {g}&&\text{ for all }[a,a+n^{-1}]\in \mathcal{I}.\end{align}
and with probability one  
\begin{equation}\label{conccond2}
\frac{nr}{{g}}\sum_{[a,a+n^{-1}]\in \mathcal{I}}\mathbb{E} X_a^2\le \Big|\{k:\lambda_{k,0}\in [E_0-r-3\sqrt{t},E_0+r+3\sqrt{t}]\}\Big|^2,\end{equation}
furthermore, with probability one,
\begin{equation}\label{conccond4} \frac{{|\{k\,:\,\lambda_{k,0}\in [E_0-r+n^{-1},E_0+r-n^{-1}]\}|}}{|\{k\,:\,\lambda_{k,0}\in [E_0-r-3\sqrt{t},E_0+r+3\sqrt{t}]\}|}\ge 1-c_n,\end{equation}
for a deterministic $c_n$ tending to zero.

Let $K_2=\{k\,:\, \lambda_{k,t}\in [E_0-r,E_0+r]\}.$

\begin{lemma}\label{rigenconc}
Choose $t=t_n$ such that $t=o(1/(n\log^2(n)))$ and $t\ge e^{-\sqrt{n}}$. Under the assumptions above, there is a $c$ with the following property. Let $h>0$, then for any large enough $n$, we can find random sets $J_k\subset \{1,2,\dots,n\}$ such that $|J_k|\le 3h$, and
\[\mathbb{E}\frac{1}{|K_2|}\sum_{k\in K_2} \sum_{j\in J_k} |u_{k,t}(j)|^2\ge 1-\frac{c}{\sqrt{h}}.\]
\end{lemma}
\begin{proof}

We say that an interval $[a,a+n^{-1}]\in \mathcal{I}$ bad if at least one of the following is true:
\begin{itemize}
    \item $X_a>h$,
    \item $\hat{m}_0(a+in^{-1})>\frac{(nt)^{-1/2}}4 $,
    \item $\hat{m}_0(a+n^{-1}+in^{-1})>\frac{(nt)^{-1/2}}4 $.
\end{itemize}
Let $\mathcal{I}_B$ be the (random) set of bad intervals in $\mathcal{I}$.

For any $[a,a+n^{-1}]\in \mathcal{I}$, it follows from conditions~\eqref{conccond1},~\eqref{conccond3} and Markov's inequality that
\[\mathbb{P}([a,a+n^{-1}]\text{ is bad})\le \frac{{g}}{h}+\frac{4{g}\log n}{(nt)^{-1/2}}+\frac{4{g}\log n}{(nt)^{-1/2}}\le \frac{2{g}}{h}\]
for all large enough $n$. From now on we assume that $n$ is large enough.

Thus, using the CSB inequality, we see that
\begin{equation}\label{badcard1}\mathbb{E} X_a\mathbbm{1}([a,a+n^{-1}]\text{ is bad}) \le \sqrt{\mathbb{P}([a,a+n^{-1}]\text{ is bad}) \mathbb{E} X_a^2}\le \sqrt{\frac{2{g}}{h}} \sqrt{\mathbb{E} X_a^2}.
\end{equation}

Let
$K_3=\{k\,:\,\lambda_{k,0}\in [E_0-r-3\sqrt{t},E_0+r+3\sqrt{t}]\},$ and
\[Z=\sum_{[a,a+n^{-1}]\in \mathcal{I}_B} X_a=\sum_{[a,a+n^{-1}]\in \mathcal{I}}X_a\mathbbm{1}([a,a+n^{-1}]\text{ is bad}).\]

Using \eqref{badcard1}, and then the CSB inequality, we see that
\[\mathbb{E} Z\le \sqrt{\frac{2{g}}{h}} \sum_{[a,a+n^{-1}]\in \mathcal{I}} \sqrt{\mathbb{E} X_a^2}\le \sqrt{\frac{2{g}}{h}} \sqrt{|\mathcal{I}|} \sqrt{\sum_{[a,a+n^{-1}]\in \mathcal{I}} \mathbb{E} X_a^2}.\]

Here the right hand side can be bounded using condition~\eqref{conccond2} to obtain that 
\begin{align}\label{EZbecs}\mathbb{E} Z\le \sqrt{\frac{2{g}}{h}} \sqrt{|\mathcal{I}|} \sqrt{\sum_{[a,a+n^{-1}]\in \mathcal{I}} \mathbb{E} X_a^2}\le \sqrt{\frac{2{g}}{h}} \sqrt{\frac{|\mathcal{I}|}{nr}}\sqrt{{g}} |K_3|\le 
\frac{2{g}}{\sqrt{h}}  |K_3|
\end{align}
with probability $1$.

We set $E_\ell=E_0-r$ and $E_u=E_0+r$.

Let 
\[K=\{j\,:\, \lambda_{j,0}\in [E_\ell+n^{-1},E_u-n^{-1}]\setminus \cup_{[a,a+n^{-1}]\in \mathcal{I}_B} [a,a+n^{-1}]\}.\]

For $j\in K$, we have that $\lambda_{j,0}\in [a,a+n^{-1}]$ for some $[a,a+n^{-1}]\in \mathcal{I}\setminus\mathcal{I}_B$. Then
\[|m_0(\lambda_{j,0}+in^{-1})|\le \hat{m}_0(\lambda_{j,0}+in^{-1})\le 2\left(\hat{m}_0(a+in^{-1})+\hat{m}_0(a+n^{-1}+in^{-1})\right)\le (nt)^{-1/2}. \]

As before, let 
\[J_k=\{j\in K\,:\, |\lambda_{k,t}-\lambda_{j,0}|< n^{-1}\}.\]

Then Theorem~\ref{lemmaconc01} gives us that
\begin{equation}\label{conceq11}
    \mathbb{E}\frac{1}{|K_2|}\sum_{k\in K_2} \sum_{j\in J_k} |u_{k,t}(j)|^2 \ge\mathbb{E} \frac{|K|}{|K_3|}\left(1-O(\sqrt{nt})\right).
\end{equation}

Let $K_0=\{j\,:\, \lambda_{j,0}\in [E_\ell+n^{-1},E_u-n^{-1}]\}$. Note that 
\[\mathbb{E}\frac{|K|}{|K_3|}=  \mathbb{E}\frac{|K_0|}{|K_3|}-\mathbb{E}\frac{Z}{ |K_3|}\ge 
\mathbb{E}\frac{|K_0|}{|K_3|}-\frac{2{g}}{\sqrt{h}}\ge 1-\frac{3{g}}{\sqrt{h}}  
\]
for all large enough $n$, where the second inequality follows from \eqref{EZbecs}, and the last inequality follows from condition \eqref{conccond4}. It is also clear that any function that is $O(\sqrt{nt})$ is at most $\frac{1}{\sqrt{h}}$ for any large enough $n$.

Inserting these into \eqref{conceq11}, we see that for all large enough $n$, we have
\begin{align*}\mathbb{E}\frac{1}{|K_2|}\sum_{k\in K_2} \sum_{j\in J_k} |u_{k,t}(j)|^2\ge \left(1-\frac{3{g}}{\sqrt{h}} \right)\left(1-\frac{1}{\sqrt{h}}\right)\ge 1-\frac{c}{\sqrt{h}}
\end{align*}
for some constant not depending on $h$.

Finally, observe that $|J_k|\le 3h$. Indeed, the interval $[\lambda_{k,t}-n^{-1},\lambda_{k,t}+n^{-1}]$ can intersect at most three of the intervals in $\mathcal{I}\setminus\mathcal{I}_B$, each of these intervals can contain at most $h$ of the numbers $\lambda_{j,0}$.
\end{proof}

\subsection{Application to the discrete torus}
 Let $(\lambda_{1,0},\lambda_{2,0},\dots,\lambda_{n^2,0})$ be the eigenvalues of $A_{C,D}$, where $(C,D)$ chosen as uniform random element of $[0,1]^2$. Let us choose $E_0\in (-4,0)\cup (0,4)$, $r=r_n=n^{-\delta}$, where $0<\delta<1$.

We now check that the conditions of Lemma~\ref{rigenconc} hold in this case. Using Lemma~\ref{expdensity}, we see that to verify conditions \eqref{conccond1} and \eqref{conccond3}, we need to prove that for a large enough ${g}$ (not depending on $n$), we have 
\begin{align}\label{conccond1s}\int_{-4}^4 \frac{1}{\sqrt{(x-\lambda)^2+(n^2)^{-2}}} \varrho_{0,0}(x)dx&\le {g}\log n&&\text{ for all }\lambda\in [E_0-r,E_0+r], \text{ and}\\
\label{conccond3s}n^2\int_{a}^{a+n^{-2}} \varrho_{0,0}(x)dx+1&\le {g}&& \text{ for all } [a,a+n^{-2}]\in \mathcal{I}.
\end{align}

As $\varrho_{0,0}(x)$ is uniformly bounded  on the interval $[E_0-r,E_0+r]$ for all large enough $n$, condition~\eqref{conccond3s} follows.  

To prove \eqref{conccond1s}, let us choose $\varepsilon>0$ such that $[E_0-2\varepsilon,E_0+2\varepsilon]\subset (-4,0)\cup (0,4)$. On the interval $[E_0-2\varepsilon,E_0+2\varepsilon]$, the function $\varrho_{0,0}$ is bounded
by some constant $c$. Then for $\lambda\in [E_0-\varepsilon,E_0+\varepsilon]$, we have
\begin{align*}
    \int_{-4}^4 \frac{1}{\sqrt{(x-\lambda)^2+(n^2)^{-2}}} \varrho_{0,0}(x)dx&\le  \int_{\lambda-\varepsilon}^{\lambda+\varepsilon} \frac{1}{\sqrt{(x-\lambda)^2+(n^2)^{-2}}} \varrho_{0,0}(x)dx+\frac{1}{\sqrt{\varepsilon^2+(n^2)^{-2}}}\\&\le c\int_{-\varepsilon}^{\varepsilon} \frac{\sqrt{2}}{|t|+n^{-2}} dt+\frac{1}{\varepsilon}\\&\le {g}\log n
\end{align*}
for a large enough ${g}$.  As $[E_0-r,E_0+r]\subset [E_0-\varepsilon,E_0+\varepsilon]$ for all large enough $n$, condition~\eqref{conccond1s} follows.

Using Lemma~\ref{detest2}, we see that there is a $c>0$ only depending on $E_0$ such that 
\[\mu_{C,0}*\mu_{D,0}([E_0-r-n^{-1},E_0+r+n^{-1}])\ge cn^2r\]
for all $(C,D)$. Combining this with our estimate on the number of close pairs given in Proposition~\ref{Lemmaclosepairs}, condition \eqref{conccond2} follows.

Conditions \eqref{conccond4}  can be verified using Lemma~\ref{detest2}. 

Therefore, Lemma~\ref{rigenconc} can be applied to give  the following theorem.

\begin{theorem}\label{Torusconc}
Let $\gamma>1$, and set $t=n^{-2\gamma}$. Let us choose $E\in (-4,0)\cup (0,4)$, and let $r=n^{-\delta}$, where $0<\delta<1$.  Let $(\lambda_{1,t},u_{1,t}),(\lambda_{1,t},u_{1,t}),\dots,(\lambda_{{n^2},t},u_{{n^2},t})$ be the eigenvalue-eigenvectors pairs of $A_{C,D}+\sqrt{t}{W}_{n^2}$ such that $\lambda_{1,t}\le \lambda_{2,t}\le\dots\le \lambda_{n^2,t}$ and $\|u_{i,t}\|_2=n$. Write $u_{i,t}$ in the bases $(u_{1,0},u_{2,0},\dots,u_{{n^2},0})$ to obtain a vector $v_i\in \mathbb{C}^{n^2}$. Let
\[K_2=\{k\,:\, \lambda_{k,t}\in [E-r,E+r]\}.\]

There is a $c$ with the following property. Let $h>0$, then for any large enough $n$, we can find random sets $J_k\subset \{1,2,\dots,n^2\}$ such that $|J_k|\le 3h$, and
\[\mathbb{E}\frac{1}{|K_2|}\sum_{k\in K_2}\sum_{j\in J_k} |v_k(j)|^2\ge 1-\frac{c}{\sqrt{h}}.\]
\end{theorem}

\begin{corollary}\label{CorTorusconc}
With the notations of Theorem \ref{Torusconc}, let $k$ be a uniform random element of $K_2$. There is a $\varepsilon>0$ with the property that for all large enough $n$, we have
\[\mathbb{P}(|v_k(j)|>\varepsilon\text{ for some j})>\frac{1}2.\]
\end{corollary}

\subsection{Properties of concentrated vectors}

We understand that in the product phase an eigenvector gets most of its weight from a few original eigenvectors. However, this does not directly describe its local structure. Small contributions from other eigenvectors could have large local influence.

In this section, we show that the largest local Fourier coefficient  distinguishes the local empirical distribution of eigenvectors in the product phase from Gaussian waves. 

\medskip

{\noindent \bf Local Fourier transform.} Given a positive integer $\ell$ and a vector $u\in \mathbb{C}^{\mathbb{Z}^2}$, we define the vector $\hat{u}(\ell,\cdot,\cdot)\in \mathbb{C}^{\{0,1,\dots,\ell-1\}^2}$ by
\[\hat{u}(\ell,{\rm{s}},{\rm{t}})=\frac{1}\ell\sum_{x=0}^{\ell-1}\sum_{y=0}^{\ell-1}\exp\left(-\frac{2\pi i}\ell(x {\rm{s}}+y {\rm{t}}) \right) u(x,y).\]
Given a vector $u\in \mathbb{C}^{\{1,2,\dots,n\}^2}$, we define $\underline{u}$ by $\underline{u}(x,y)=u((x,y)+o_n)$. By appending it with zeros, we can consider $\underline{u}$ as a vector in $\mathbb{C}^{\mathbb{Z}^2}$.

The eigenvectors $u$ and the Gaussian wave $Z$ both have order one entries. However, their local Fourier coefficients are of different order. 

For comparison, it helps to consider a third random vector: the vector a i.i.d.\ standard Gaussian entries. For such a vector the typical local Fourier coefficient would be of order one as well. 

Since the Gaussian wave is a generalized eigenfunction,  its local Fourier transform concentrates  about evenly on the coefficients $({\rm{s}},{\rm{t}})$ such that
$$
2\cos(2\pi {\rm{s}}/\ell)+2\cos(2\pi {\rm{t}}/\ell)\approx E.
$$
Since there are order $\ell$ such coefficients,  each one has norm about $\sqrt{\ell}$. 

In contrast, for the product phase eigenvector $u$, the local Fourier transform $\hat u$ should have a few dominant coefficients with norm of order $\ell$. 

The non-convergence part of Theorem~\ref{t:main} follows from the following proposition.

\begin{proposition}\label{propnonconv}\hfill
\begin{enumerate}[(i)]
    \item 
    Using the notation of Theorem~\ref{t:main}, let $\gamma>1$ and $\delta<1$, and let $k$ be a uniform random element of $\{k\,:\,\lambda_{k,t}\in [E\pm n^{-\delta}]\}$. Then there is a $c>0$ so that for every $\ell$ and all large enough $n$ we have 
\[\mathbb{P}\left(|\hat{\underline{u}}_{k,t}(\ell,{\rm{s}},{\rm{t}})|\ge c \ell\text{ for some }{\rm{s}},{\rm{t}}\right)>\frac{1}4.\]
    \item 
    For any constant $\varepsilon>0$, we have
\[\lim_{\ell\to \infty}\mathbb{P}\left(|\hat{Z}(\ell,{\rm{s}},{\rm{t}})|\ge  \ell^{1/2+\varepsilon}\text{ for some }{\rm{s}},{\rm{t}}\right) =0.\]
\end{enumerate}
\end{proposition}

To prove Proposition \ref{propnonconv}, we first need a few lemmas. Throughout this section we assume that $E>0$. The proof in the case of $E<0$ is essentially the same.

\begin{lemma}\label{nagyfourier}
Let
\[w_{k,j}=\left(\exp\left(2\pi i\left(x\left(\frac{k}n+{\rm{c}}\right)+y\left(\frac{j}n+{\rm{d}}\right)\right)\right)\right)_{(x,y)\in \{1,2,\dots,n\}^2}\]
be an eigenvector of $A_{{\rm{c}},{\rm{d}}}$ of norm $n$. Assume that $(o_n)_1,(o_n)_2\le n-\ell+1$.  Then there are ${\rm{s}},{\rm{t}}\in \{0,1,\dots,\ell-1\}$ such that
\[|\hat{\underline{w}}_{k,j}(\ell,{\rm{s}},{\rm{t}})|\ge \frac{\ell}{25}.\]
\end{lemma}
\begin{proof}
We can choose ${\rm{s}}$ and ${\rm{t}}$ in $\{0,1,\dots,\ell-1\}$ such that ${\rm{s}}'=\frac{k}n+{\rm{c}}-\frac{{\rm{s}}}\ell$ and ${\rm{t}}'=\frac{j}n+{\rm{d}}-\frac{{\rm{t}}}\ell$ are both at distance at most $\frac{1}{2\ell}$ from the closest integer. Then
\begin{align*}\hat{\underline{w}}_{k,j}(\ell,{\rm{s}},{\rm{t}})=\frac{1}{\ell}\exp\left(2\pi i\left((o_n)_1\left(\frac{k}n+{\rm{c}}\right)+(o_n)_2\left(\frac{j}n+{\rm{d}}\right)\right)\right)\sum_{x=0}^{\ell-1}\sum_{y=0}^{\ell-1} \exp(2\pi i({\rm{s}}'x+{\rm{t}}'y)).
\end{align*}
Thus,
\begin{equation}\label{lfle}|\hat{\underline{w}}_{k,j}(\ell,{\rm{s}},{\rm{t}})|=\frac{1}{\ell} \left|\frac{\exp(2\pi i \ell{\rm{s}}')-1}{\exp(2\pi i {\rm{s}}')-1}\right|\left|\frac{\exp(2\pi i \ell{\rm{t}}')-1}{\exp(2\pi i {\rm{t}}')-1}\right|,\end{equation}
provided that ${\rm{s}}'$ and ${\rm{t}}'$ are not integers, which we will assume. The case when one of these numbers is an integer can be handled easily. We may also assume that $|{\rm{s}}'|\le \frac{1}{2\ell}$ and $|{\rm{t}}'|\le \frac{1}{2\ell}$, by the choice of ${\rm{s}}$ and ${\rm{t}}$, and the fact that \eqref{lfle} is unaffected by adding an integer to ${\rm{s}}'$ or ${\rm{t}}'$. Using that facts that for $u\in [-\pi,\pi]$, we have $0.2|u|\le |e^{iu}-1|\le |u|$, Lemma~\ref{nagyfourier} follows.
\end{proof}


Note that $\hat{Z}(\ell,{\rm{s}},{\rm{t}})$ is Gaussian. Let us denote its variance by $\Var(\ell,{\rm{s}},{\rm{t}})$. 

Let $\theta$ be a random variable with density $\frac{d_0(u)d_0(E-u)}{\varrho_{0,0}(E)}$. Let $$(A,B)=(\pm \cos^{-1}(\theta/2),\pm \cos^{-1}((E-\theta)/2)),$$ where the signs are chosen uniformly at random, independently from each other and $\theta$.

Let $s_\alpha = \alpha -\frac{2\pi}{\ell}{\rm{s}}$ and $t_\beta = \beta -\frac{2\pi}{\ell}{\rm{t}}$.

For any $r\in\mathbb{R}$, let $\|r\|_{\text{mod }2\pi}$ be the distance of $r$ from the set $2\pi \mathbb{Z}$. For $(r_1,r_2)\in \mathbb{R}^2$, let $\|(r_1,r_2)\|_{\infty,\text{mod } 2\pi}=\max\left(\|r_1\|_{\text{mod } 2\pi},\|r_2\|_{\text{mod }2\pi}\right)$.

\begin{lemma}\label{varstbound}
We have
\begin{equation}
\Var(\ell,{\rm{s}},{\rm{t}})\le \mathbb{E} \left(\min\left(\ell,4/{\left\|({\rm{s}}_A,{\rm{t}}_B)\right\|_{\infty,\text{mod }2\pi}}\right)\right)^2.\end{equation}

\end{lemma}
\begin{proof}
Observe that $\Var(\ell,{\rm{s}},{\rm{t}})$ is given by
\[
\frac{1}{\ell^2}\sum_{x_1=0}^{\ell-1}\sum_{y_1=0}^{\ell-1}\sum_{x_2=0}^{\ell-1}\sum_{y_2=0}^{\ell-1} \exp\left(-\frac{2\pi i}\ell(x_1{\rm{s}}+y_1{\rm{t}}) \right) \frac{\varrho_{x_1-x_2,y_1-y_2}(\lambda)}{\varrho_{0,0}(\lambda)} \exp\left(\frac{2\pi i}\ell(x_2{\rm{s}}+y_2{\rm{t}}) \right).
\]


 We can rephrase \eqref{varrhouvformula}, as
\[\frac{\varrho_{\ux,\uy}(E)}{\varrho_{0,0}(E)}=\mathbb{E} \exp(i(\ux A+\uy B)).\]
Thus, the variance can be rewritten as 
\begin{equation}\label{varstbound1}\Var(\ell,{\rm{s}},{\rm{t}})=\frac{1}{\ell^2}\mathbb{E} \left|\sum_{x=0}^{\ell-1} \exp\left(ix{\rm{s}}_A\right)\right|^2\left|\sum_{y=0}^{\ell-1}\exp \left(iy{\rm{t}}_B \right)\right|^2. 
\end{equation}

 Note that
\begin{equation}\label{varstbound2}
    \left|\sum_{x=0}^{\ell-1} \exp\left(ix{\rm{s}}_A\right)\right|\le \min\left(\ell,1/{\sin\left(\left\|{\rm{s}}_A\right\|_{\text{mod }2\pi}/2\right)}\right)\le \min\left(\ell,4/{\left\|{\rm{s}}_A\right\|_{\text{mod }2\pi}}\right).
\end{equation}

Combining \eqref{varstbound1} and \eqref{varstbound2}, the statement follows.
\end{proof}

Let $\gamma$ be the support of $(A,B)$, that is,\[\gamma=\{(\alpha,\beta)\in [-\pi,\pi]^2\,:\,2\cos(\alpha)+2\cos(\beta)=E\}.\]

Then it follows from Lemma~\ref{varstbound}, that
\begin{equation}\label{varstbound3}
\Var(\ell,{\rm{s}},{\rm{t}})\le \mathbb{E} \left(\min\left(\ell,4/\inf_{(\alpha,\beta)\in\gamma}\left\|({\rm{s}}_\alpha,{\rm{t}}_\beta)\right\|_{\infty,\text{mod }2\pi}\right)\right)^2.\end{equation}

By replacing $s$ with $s-\ell$, if $s>\ell/2$, we may assume that $|s|\le \ell/2$. We also can assume that $|t|\le \ell/2$. As a direct consequence of \eqref{varstbound3}, we obtain the following lemma.

\begin{lemma}\label{lemmakozel}
We have $\Var(\ell,{\rm{s}},{\rm{t}})\le 16\ell$, unless $\left(\frac{2\pi{\rm{s}}}{\ell},\frac{2\pi{\rm{s}}}{\ell}\right)\in B_{\ell^{-1/2}}(\gamma)$. Here $B_{\ell^{-1/2}}(\gamma)$ is the $\ell^{-1/2}$ neighborhood of $\gamma$, that is,  
\[B_{\varepsilon}(\ell^{-1/2})=\{(\alpha',\beta')\in [-\pi,\pi]^2\,:\,\|(\alpha',\beta')-(\alpha,\beta)\|_{\infty,\text{mod }2\pi}\le \ell^{-1/2}\text{ for some }(\alpha,\beta)\in \gamma\}.\]

\end{lemma}

 Recalling that we assumed $E>0$, we see that the curve $\gamma$ is contained in the square  \break $\left[-\cos^{-1}((E-2)/2),\cos^{-1}((E-2)/2)\right]^2$ and there is a neighbour of the vertices of the square above which does not intersect $\gamma$. Thus, we have the following lemma.
\begin{lemma}\label{ledelta}
There is a $0<\delta<\cos^{-1}((E-2)/2)$ such that for any large enough $\ell$, we have that if  $(\alpha',\beta')\in B_{\ell^{-1/2}}(\gamma)$, then $|\alpha'|\le \delta$ or $|\beta'|\le \delta$. 
\end{lemma}


We set $\varepsilon=(\cos^{-1}((E-2)/2)-\delta)/2$.

\begin{lemma}\label{alphadensity}
There is a $d$  such that density of $A$ is bounded by $d$ on the interval $[-\delta-\varepsilon,\delta+\varepsilon]$. Assuming that $\left|\frac{2\pi{\rm{s}}}{\ell}\right|\le \delta$, the density of ${\rm{s}}_A$ is also bounded by $d$ on the interval $[-\varepsilon,\varepsilon]$.
\end{lemma}
\begin{proof}
Note that $A$ is supported on $\left[-\cos^{-1}((E-2)/2),\cos^{-1}((E-2)/2)\right]$. By a change of variables, the density of $A$ is given by \[\frac{|\sin(\alpha)|d_0(2\cos(\alpha))d_0(E-2\cos(\alpha))}{\varrho_{0,0}(E)}=\frac{d_0(E-2\cos(\alpha))}{2\pi\varrho_{0,0}(E)}\]
on the support of $A$. This density is continuous on $\left(-\cos^{-1}((E-2)/2),\cos^{-1}((E-2)/2)\right)$, so the first statement follows. The second statement is straightforward from the first one. 
\end{proof}

Now we are able to prove the following strong estimate on $\Var(\ell,{\rm{s}},{\rm{t}})$.

\begin{lemma}\label{hatZVar}
There is $c$ depending only on $E$ such that $\Var(\ell,{\rm{s}},{\rm{t}})\le c\ell.$
\end{lemma}
\begin{proof}
Using Lemma \ref{lemmakozel}, we may assume that $\left(\frac{2\pi{\rm{s}}}{\ell},\frac{2\pi{\rm{s}}}{\ell}\right)\in B_{\ell^{-1/2}}(\gamma)$. Then Lemma~\ref{ledelta} gives us that $\left|\frac{2\pi{\rm{s}}}{\ell}\right|\le \delta$ or $\left|\frac{2\pi{\rm{t}}}{\ell}\right|\le \delta$. Without the loss of generality, we may assume  that $\left|\frac{2\pi{\rm{s}}}{\ell}\right|\le \delta$.

Assume that $\ell$ is large enough so that $\ell^{-1}<\varepsilon$.  Using Lemma \ref{alphadensity}, we get
\begin{align*}
    \mathbb{E}\left[\mathbbm{1}(|{\rm{s}}_A|\le \ell^{-1})\left(\min\left(\ell,4/{\left\|{\rm{s}}_A,{\rm{t}}_B\right\|_{\infty,\text{mod }2\pi}}\right)\right)^2\right]&\le 2d\ell,\\
    \mathbb{E}\left[\mathbbm{1}(\ell^{-1}<|{\rm{s}}_A|\le \varepsilon)\left(\min\left(\ell,4/{\left\|{\rm{s}}_A,{\rm{t}}_B\right\|_{\infty,\text{mod }2\pi}}\right)\right)^2\right]&\le 2d\int_{\ell^{-1}}^\varepsilon \frac{16}{u^2} du\le 32d\ell,\\
    \mathbb{E}\left[\mathbbm{1}(\varepsilon<|{\rm{s}}_A|)\left(\min\left(\ell,4/{\left\|{\rm{s}}_A,{\rm{t}}_B\right\|_{\infty,\text{mod }2\pi}}\right)\right)^2\right]&\le 16\varepsilon^{-2}.
\end{align*}
Summing these inequalities, and Lemma \ref{varstbound}, the statement follows.
\end{proof}


Now we are ready to prove Proposition \ref{propnonconv}.
\begin{proof}[Proof of Proposition \ref{propnonconv}]
Let $\varepsilon$ be provided by Corollary~\ref{CorTorusconc}, and let $\varepsilon'=\frac{\varepsilon}{25}$. Let $j$ be chosen such that $|v_k(j)|$ is maximal. With probability at least $\frac{1}2$, we have $|v_k(j)|>\varepsilon$. Assume that this event occurs.   Let ${\rm{s}}$ and ${\rm{t}}$ be such that $|\hat{\underline{u}}_{j,0}(\ell,{\rm{s}},{\rm{t}})|\ge \frac{\ell}{25}$, such ${\rm{s}}$ and ${\rm{t}}$ exist by Lemma~\ref{nagyfourier}. By the unitary invariance of the noise, conditioned on $(C,D)$, $k$ and  $|v_k(1)|,|v_k(2)|,\dots,|v_k(n^2)|$ the phases of the components of  $v_k$ are i.i.d. uniform random. Thus, the phase of $v_k(j)\hat{\underline{u}}_{j,0}(\ell,{\rm{s}},{\rm{t}})$ and $\sum_{m\neq j} v_k(m)\hat{\underline{u}}_{m,0}(\ell,{\rm{s}},{\rm{t}})$ are independent uniform random. Thus, by symmetry, we have
\[\mathbb{P}\left(\Re \left( \overline{v_k(j)\hat{\underline{u}}_{j,0}(\ell,{\rm{s}},{\rm{t}})} \sum_{m\neq j} v_k(m)\hat{\underline{u}}_{m,0}(\ell,{\rm{s}},{\rm{t}})\right)\ge 0\right)\ge \frac{1}2.\]
On this event,
\[|\hat{\underline{u}}_{k,t}(\ell,{\rm{s}},{\rm{t}})|=\left|\sum_{m=1}^{n^2} v_k(m)\hat{\underline{u}}_{\ell,0}(\ell,{\rm{s}},{\rm{t}})\right|\ge |v_k(j)\hat{\underline{u}}_{j,0}(\ell,{\rm{s}},{\rm{t}})|\ge {\varepsilon}\frac{\ell}{25}.\]

This concludes the proof of the first part of the proposition. 

Combining Lemma~\ref{hatZVar}, the Gaussian tail estimate and the union bound
the second part of the proposition follows. 
\end{proof}
\appendix
\section{The proof of some technical statements}

\subsection{The proof of Lemma~\ref{newcontinuity}}
Write 
$
d_\ux=a_\ux+b_\ux,
$
where $a_\ux$ is the unique continuous function that is zero outside $[-2,2]$, agrees with $d_\ux$ on $[-2+\epsilon,2-\epsilon]$ and is linear both on $[-2,\epsilon-2]$ and on $[2-\epsilon,2]$. 

Since $T_{|\ux|}$ has absolute value at most 1 and absolute derivative at most $u^2$, by examining the formula \eqref{e:arcsine}, we see that 
$$
 \|a_\ux'\|_\infty \le \epsilon^{-3/2}(1+\ux^2),\qquad \|b_\ux\|_1 \le \epsilon^{1/2}.$$

Since $b_\ux*b_\uy$ vanishes on $[3\epsilon-4,-3\epsilon]\cup[3\epsilon,4-3\epsilon]$, for any $\lambda\in [3\epsilon-4,-3\epsilon]\cup[3\epsilon,4-3\epsilon]$, we have
\[
    |(d_\ux*d_\uy)'(\lambda)|=|(d_\ux*a_\uy+a_\ux*b_\uy)'(\lambda)|\le \|d_\ux*a_\uy'\|_\infty+\|a_\ux'*b_\uy\|_\infty. 
\]

Using Young's inequality, we get that
\begin{equation}\label{dudvdiffbecs}
\|d_\ux*a_\uy'\|_\infty+\|a_\ux'*b_\uy\|_\infty\le \|d_\ux\|_1\|a_\uy'\|_\infty+\|b_\uy\|_1\|a_\ux'\|_\infty\le 2\varepsilon^{-3/2}(1+\ux^2+\uy^2).
\end{equation}

Thus, the first statement follows.

We write
$$
|d_\ux*d_\uy*(\kappa_\eta-\pi \delta)(\lambda)|= 
|(b_\ux*b_\uy+d_\ux*a_\uy+a_\ux*b_\uy)*(\kappa_\eta-\pi \delta)(\lambda)|,
$$
we bound the terms separately. First
$$
|b_\ux*b_\uy*(\kappa_\eta-\pi \delta)(\lambda)|=|b_\ux*b_\uy*\kappa_\eta(\lambda)|.
$$
The distance of $\lambda$ from the support of $b_\ux*b_\uy$ is at least $\epsilon$. Thus the convolution above only uses the values of $\kappa_\eta$ on the set $J=(-\epsilon,\epsilon)^c$, and we can replace 
$\kappa_\eta$ by $1_J\kappa_\eta$. By Young's inequality, we get the upper bound
\begin{equation}\label{e:edge}
\|b_\ux*b_\uy*(1_J\kappa_\eta)\|_\infty\le \|b_\ux\|_1 \|b_\uy\|_1 \|1_J\kappa_\eta\|_\infty\le \epsilon \frac\eta{\eta^2+\epsilon^2}.
\end{equation}
Let $f_\eta(t)=\int_{-\infty}^t (\kappa_\eta-\pi\delta)(s)ds$. Then $f_1$ decays like $1/|t|$ at infinity, so  
as $p>1$, by scaling, $$\|f_\eta\|_p=\eta^{1/p}\|f_1\|_p <\infty
.$$
Let $q$ satisfy $1/p+1/q=1$. We move the derivative and use Young's inequality to get
\begin{align*}
\|(d_\ux*a_\uy+a_\ux*b_\uy)*(\kappa_\eta-\pi \delta)\|_\infty&= \|(d_\ux*a_\uy'+a_\ux'*b_\uy)*f_\eta\|_\infty\\&\le\|f_\eta\|_p(\|d_\ux*a_\uy'\|_{q}+\|a_\ux'*b_\uy\|_q).
\end{align*}
Combining the fact that $a'_\ux, a'_\uy$ are supported on an interval of length 4 and \eqref{dudvdiffbecs}, we get
\[\|d_\ux*a_\uy'\|_{q}+\|a_\ux'*b_\uy\|_q\le 4^{1/q} (\|d_\ux*a_\uy'\|_\infty+\|a_\ux'*b_\uy\|_\infty)\le 8 \varepsilon^{-3/2}(1+\ux^2+\uy^2).\]

The second claim follows with $c_p=8\|f_1\|_p$. The last statement follows from the first two and the mean value theorem by enlarging $c_p$ if necessary.

\subsection{The proof of Lemma~\ref{varrhosing}}

By symmetry, we may assume that $x>0$.  First assume that $0<x<2$. We write $d_0=a+b$, where $a(t)=\mathbbm{1}\left(|t|\le 2-x/2\right)d_0(t)$ and $b(t)=\mathbbm{1}\left(|t|> 2-x/2\right)d_0(t)$. Observing that $b*b(x)=0$, we see that
\begin{multline*}d_0*d_0(x)\le \|a*a\|_{\infty}+2\|a*b\|_\infty\le \|a\|_2^2+2\|a\|_\infty \|b\|_1\\\le \log(4/x)+2(\pi^{-1}\sqrt{2/x})(4\pi^{-1}\sqrt{x/2})\le \log(100/x). 
\end{multline*}

Assuming that $2\le x\le 4$, we have
\[
   (d_0*d_0)(x)=\int_{-2+x}^2 d_0(t)d_0(x-t)dt\le \int _{-2+x}^2 \frac{1}{\pi\sqrt{(2-t)(2-x+t)}} dt=1<\log(100/x).\qedhere
\]

\subsection{The proof of Lemma~\ref{stieltjescont1G}}

Differentiating under the integral sign gives $|f'|\le g/\eta$. Thus $|f-z|$ is $\frac{\max_J g}{\eta}$-Lipschitz on $J$. We may assume that $\max_J|f-z|\ge b$, otherwise the statement is trivial. We set \[\ell=\min\left(\eta\frac{\max_J |f-z|-b}{\max_J g},{r}\right)
.\]Without loss of generality, assume that $1_J|f-z|$ is maximized at some $\lambda_0$ in the left half of~$J$. Then for $\lambda\in [\lambda_0,\lambda_0+\ell]\subset J$, we have
\[|f(\lambda)-z|\ge |f(\lambda_0)-z|-\frac{\max_J g}{\eta}(\lambda-\lambda_0)\ge b.\]

This proves the first statement, the proof of the second statement is very similar, we prove the third one. We may assume that $\min_J f\le b$, otherwise the statement is trivial. We set \[\ell=\min\left(\eta\frac{b-\min_J f }{b},{r}\right)
.\]Without loss of generality, assume that $1_J f$ is minimized at some $\lambda_0$ in the left half of $J$. Then  $[\lambda_0,\lambda_0+\ell]\subset J$. Let $\lambda\ge \lambda_0$ be the smallest point such that $f(\lambda)=b$ if there is any. It is enough to prove that $\lambda\ge \lambda_0+\ell$. We prove by contradiction. So assume that $\lambda-\lambda_0<\ell\le \eta\frac{b-\min_J f }{b}$. Note that on the interval $[\lambda_0,\lambda]$, we have $f'\le f/\eta\le b/\eta$. Thus,
\[f(\lambda)\le f(\lambda_0)+\frac{b}{\eta} (\lambda-\lambda_0)< b,\]
 which is a contradiction.

\subsection{The proof of Lemma~\ref{convergeramdomI}}

We start by proving the first statement in the special case $y=x$. Let $\delta>0$. Take any $I\in\mathcal{I}$ such that
\[\sup_{\lambda\in I^F}\left|B(\eta,\lambda,x,x)-M(x,x)\right|>2\delta.\]
We apply Lemma~\ref{stieltjescont1G} with the choice of $J=I^F$, $\nu=\mu_{n,x,x}$, $b=\delta$ and $z=M(x,x)$. Note that since  $\mu_{n,x,x}$ is a probability measure, we have $g=f$, and $\max_J g\le z+\max_J |f-z|$. Thus, 
\begin{align*}\int_{J} \mathbbm{1}\left(|f(\lambda)-z|\ge \delta\right)d\lambda\ge \min\left(\eta\frac{\max_J |f-z|-\delta}{\max_J g},{r}\right)\ge \min\left(\eta\frac{\max_J |f-z|-\delta}{\max_J |f-z|+z},{r}\right).
\end{align*}
Recall that we assumed that  $\max_{J} |f-z|\ge 2\delta$. Therefore, using the monotonicity of the map $x\mapsto \frac{x-\delta}{x+z}$, we get that 
\begin{align*}
\int_{J} \mathbbm{1}\left(|f(\lambda)-z|\ge \delta\right)d\lambda\ge \min\left(\eta\frac{\delta}{2\delta+z},{r}\right)\ge q_{x,\delta} t
\end{align*}
for some $q_{x,\delta}>0$, where we used that 
$t/c<\eta<ct $.
Thus, 
\[3\int_{E-\ell}^{E+\ell} \mathbbm{1}\left(|f(\lambda)-z|\ge \delta\right)d\lambda\ge q_{x,\delta} t  \left|\left\{I\in\mathcal{I}\,:\,\sup_{\lambda\in I^F}\left|B(\eta,\lambda,x,x)-M(x,x)\right|>2\delta\right\}\right|.
\]
The factor $3$ on the left is coming from the fact that the intervals $I^F$ ($I\in \mathcal{I}$) cover each point at most $3$ times.

Rearranging 
\[
    \frac{6\ell}{q_{x,\delta} t|\mathcal{I}|}  \mathbb{P}_{\lambda}\left( \left|B(\eta,\lambda,x,x)-M(x,x)\right|\ge\delta\right)\ge \mathbb{P}_I\left(\sup_{\lambda\in I^F}\left|B(\eta,\lambda,x,x)-M(x,x)\right|>2\delta\right).
\]
Noting that $|\mathcal{I}|=\frac{2\ell}{r}-1=\frac{2\ell}{tn^\varepsilon}-1$ and using condition \eqref{limitMA}, it follows that the left hand side converge to zero for all choices $\delta>0$. So the first statement follows in the special case $y=x$.  

Now we prove the general case. By using the already established case, it is enough to prove that for any $\delta>0$, we have $\lim_{n\to\infty} \mathbb{P}(I\in \mathcal{I}_B)=0$, where $\mathcal{I}_B$ consists of all the intervals $I\in \mathcal{I}$ such that
$\left|B(\eta,\lambda,x,y)-M(x,y)\right|>2\delta$ for some $\lambda\in I^F$, and $\left|B(\eta,\lambda,x,x)-M(x,x)\right|\le 1$, $\left|B(\eta,\lambda,y,y)-M(y,y)\right|\le 1$ for all $\lambda\in I^F$. 
Consider an $I\in\mathcal{I}_B$.  We apply Lemma~\ref{stieltjescont1G} with the choice of $J=I^F$, $\nu=\mu_{n,x,y}$, $b=\delta$ and $z=M(x,y)$. Using \eqref{absolute}, we see that
$$\max_J g\le 1+\frac{M(x,x)+M(y,y)}{2}.$$ Since $I\in \mathcal{I}_B$, we have $\max_J |f-z|\ge 2\delta$.
\begin{align*}\int_J \mathbbm{1}\left(|f(\lambda)-z|\ge \delta\right)d\lambda\ge \min\left(\eta\frac{\max_J |f-z|-\delta}{\max_J g},{r}\right)\ge \min\left(\eta\frac{\delta}{\max_J g},{r}\right)\ge q_{x,y,\delta} t
\end{align*}
for some $q_{x,y,\delta}>0$. Thus, $\lim_{n\to\infty} \mathbb{P}_I(I\in \mathcal{I}_B)=0$ follows the same way as before.


To prove the second statement, let $I\in\mathcal{I}$ such that $\Im m(\lambda+i\eta_\ell)\le \frac{\eta_\ell}t$ for some $\lambda\in I^F$. We apply Lemma~\ref{stieltjescont1G} with the choice of $J=I$, $\nu=\mu_{n}$, $b=\frac{\eta_\ell}{t}+n^{-\varepsilon}$. Note that $b=O(1)$ and $\min_J f\le \frac{\eta_\ell}{t}$. Thus,
\[\int_J \mathbbm{1}\left(f(\lambda) \le \frac{\eta_\ell}{t}+n^{-\varepsilon}\right)d\lambda\ge \min\left(\eta\frac{(\frac{\eta_\ell}{t}+n^{-\varepsilon})-\min_J f}{b},{r}\right)\ge q  n^{-\varepsilon} t\]
for some $q>0$. Using the same argument as before, we get that \[\lim_{n\to\infty} \mathbb{P}_I\left(\Im m(\lambda+i\eta_\ell)>\frac{\eta_\ell}t\text{ for all }\lambda\in I^F\right)=1.\]
A similar argument gives us that \[\lim_{n\to\infty} \mathbb{P}_I\left( \Im m(\lambda+i\eta_u)<\frac{\eta_u}t\text{ for all }\lambda\in I^F\right)=1.\]
Thus the second statement follows. We omit the proof of the third statement.

\subsection{The proof of Proposition~\ref{prop47}}
We write $\mathbb{E}_\lambda$ for the expectation over the uniform random choice of $\lambda\in [E\pm \ell]$.

\begin{lemma}

There is a $k$ with the property that
for any \[\eta\in  \{\eta_\ell, \eta_u\}\cup \{t n^{-2\varepsilon},2t n^{-2\varepsilon},2^2t n^{-2\varepsilon},\dots,2^bt n^{-2\varepsilon}\},\]
and fixed $\ux,\uy$, we have
\[\lim_{n\to \infty} n^{4\varepsilon}\log(n) \mathbb{P}_\lambda\left(\left| \mu_{{\rm{c}},\ux}*\mu_{{\rm{d}},\uy}*\kappa_\eta(\lambda)-\pi{\varrho_{\ux,\uy}(E)}\right|\ge \frac{n^{-2\varepsilon}}2+k(1+\ux^2+\uy^2)(\sqrt{\eta}+\ell)\right)=0.\]

Moreover, for $\eta\in \{\eta_\ell, \eta_u\}$, we have
\[\lim_{n\to \infty} n^{4\varepsilon}\log(n) \mathbb{P}_\lambda\left(\left| \mu_{{\rm{c}},\ux}*\mu_{{\rm{c}},\uy}*\kappa_\eta(\lambda)-\pi{\varrho_{\ux,\uy}(E)}\right|\ge {n^{-2\varepsilon}}\right)=0.\]
\end{lemma}
\begin{proof}

Clearly,
\begin{multline*}
    \mathbb{P}_\lambda\left(\left| \mu_{{\rm{c}},\ux}*\mu_{{\rm{d}},\uy}*\kappa_\eta(\lambda)-\pi{\varrho_{\ux,\uy}(E)}\right|\ge \frac{n^{-2\varepsilon}}2+k(1+\ux^2+\uy^2)(\sqrt{\eta}+\ell)\right)\\\le \mathbb{P}_\lambda\left(\left| \mu_{{\rm{c}},\ux}*\mu_{{\rm{d}},\uy}*\kappa_\eta(\lambda)-d_\ux*d_\uy*\kappa_\eta(\lambda)\right|\ge \frac{n^{-2\varepsilon}}2\right)\\+\mathbb{P}_\lambda\left(\left|d_\ux*d_\uy*\kappa_\eta(\lambda)-\pi\varrho_{\ux,\uy}(E)\right|\ge k(1+\ux^2+\uy^2)(\sqrt{\eta}+\ell)\right).
\end{multline*}

Using Lemma~\ref{newcontinuity} we see that the  second term is $0$ for all large enough $n$ for a sufficiently large choice of $k$. By combining Markov's inequality with the assumption that $({\rm{c}},{\rm{d}})$ is nice, we can bound the first term by
\begin{align*}
     4n^{4\varepsilon} \mathbb{E}_{\lambda}\left| \mu_{{\rm{c}},\ux}*\mu_{{\rm{d}},\uy}*\kappa_\eta(\lambda)-d_\ux*d_\uy*\kappa_\eta(\lambda)\right|^2=o(n^{-4\varepsilon}\log(n)).
\end{align*}
Thus, the first statement follows. The second statement follows from the fact that for any fixed $\ux,\uy$ and $\eta\in \{\eta_\ell, \eta_u\}$, we have $k(1+\ux^2+\uy^2)(\sqrt{\eta}+\ell)<\frac{n^{-2\varepsilon}}2$ for all large enough $n$, as it follows from the choice of our parameters.
\end{proof}

Let $(x,y),(x+\ux,y+\uy)\in\mathbb{Z}^2$. Provided that $o_n+(x,y),  o_n+(x+\ux,y+uy)\in \{1,2,\dots,n^2\}$, we have
\[\frac{t}{(A_{{\rm{c}},{\rm{d}}}-\lambda)^2+\eta^2}(o_n+(x+\ux,y+\uy),o_n+(x,y))=\frac{t}{\eta} \mu_{{\rm{c}},\ux}*\mu_{{\rm{d}},\uy}*\kappa_\eta(\lambda).\]
Thus, condition \eqref{limitMA} is equivalent to the condition that for $\eta=\eta_{\ell}$ and $\eta=\eta_{u}$, $\delta>0$, $(\ux,\uy)\in \mathbb{Z}^2$, we have
\[\lim_{n\to \infty} n^{2\varepsilon} \mathbb{P}_\lambda\left(\left|\frac{t}{\eta} \mu_{{\rm{c}},\ux}*\mu_{{\rm{c}},\uy}*\kappa_\eta(\lambda)-\frac{\varrho_{\ux,\uy}(E)}{\varrho_{0,0}(E)}\right|\ge \delta\right)=0.\]

On the event $\left| \mu_{{\rm{c}},\ux}*\mu_{{\rm{c}},\uy}*\kappa_\eta(\lambda)-\pi{\varrho_{\ux,\uy}(E)}\right|\le n^{-2\varepsilon}$, we have
\begin{multline*}
\left|\frac{t}{\eta} \mu_{{\rm{c}},\ux}*\mu_{{\rm{d}},\uy}*\kappa_\eta(\lambda)-\frac{\varrho_{\ux,\uy}(E)}{\varrho_{0,0}(E)}\right|\\\le  \frac{t}{\eta}  \left| \mu_{{\rm{c}},\ux}*\mu_{{\rm{d}},\uy}*\kappa_\eta(\lambda)-\pi{\varrho_{\ux,\uy}(E)}\right|+ |\varrho_{\ux,\uy}(E)|\left|\frac{\pi t}{\eta}-\frac{1}{\varrho_{0,0}(E)}\right|\le O(n^{-2\varepsilon})+o(1).
\end{multline*}

Thus, \eqref{limitMA} follows. Observing that $\Im m(\lambda+i\eta)=\mu_{{\rm{c}},0}*\mu_{{\rm{d}},0}*\kappa_\eta(\lambda)$, a similar argument gives that condition \eqref{condregA} holds in the weaker form mentioned in Remark \ref{remarklemmaquenched1A}. Indeed, let $\eta_0=\left(\frac{\pi \varrho_{0,0}(E)}{3k}\right)^2$. Provided that $\eta=2^ft n^{-2\varepsilon}<\eta_0$, we have
\[\lim_{n\to \infty} n^{4\varepsilon}\log(n) \mathbb{P}_\lambda\left( \mu_{{\rm{c}},0}*\mu_{{\rm{c}},0}*\kappa_\eta(\lambda)\notin \left[\frac{\pi \varrho_{0,0}(E)}2,\frac{3\pi \varrho_{0,0}(E)}2\right]\right)=0.\]
If $20\ge \eta=2^ft n^{-2\varepsilon}\ge \eta_0$, then
\[\frac{\eta_0}{20^2+8^2}\le \Im m(\lambda+i\eta)\le \eta_0^{-1}\]
using simply the fact that $\|A\|\le 4$.

We move on to prove that condition \eqref{cond2A} holds. 

On the event $\left| \mu_{{\rm{c}},0}*\mu_{{\rm{d}},0}*\kappa_\eta(\lambda)-\pi{\varrho_{0,0}(E)}\right|\le n^{-2\varepsilon}$, which has probability at least $1-o(n^{-4\varepsilon})$, we see that 
\begin{align*}
 \Im m(\lambda+i\eta_\ell)-\frac{\eta_{\ell}}{t}&=  \mu_{{\rm{c}},0}*\mu_{{\rm{d}},0}*\kappa_\eta(\lambda)-\pi\varrho_{0,0}(E)+n^{-\varepsilon}\\&\ge  n^{-\varepsilon}-\left|\mu_{{\rm{c}},0}*\mu_{{\rm{c}},0}*\kappa_\eta(\lambda)-\pi\varrho_{0,0}(E)\right|\\&\ge n^{-\varepsilon}- n^{-2\varepsilon} >n^{-2\varepsilon} 
\end{align*}
for all large enough $n$. Thus, the first part condition \eqref{cond2A} follows. The second part can be proved similarly.

To prove the uniform integrability  condition, first we notice that 
\begin{align*}\mathbb{E}_\lambda \left(\Im m(\lambda+i\eta_u)\right)^2&=\mathbb{E}_\lambda(\mu_{{\rm{c}},0}*\mu_{{\rm{d}},0}*\kappa_\eta(\lambda))^2\\&\le 2\mathbb{E}_\lambda\left|\mu_{{\rm{c}},0}*\mu_{{\rm{d}},0}*\kappa_{\eta_u}(\lambda))-d_0*d_0*\kappa_{\eta_u}(\lambda)\right|^2 +2\mathbb{E}_\lambda\left|d_0*d_0*\kappa_{\eta_u}(\lambda)\right|^2.
\end{align*}
Combining the fact that $({\rm c},{\rm d})$ is nice and Lemma~\ref{newcontinuity}, we see that the right hand side is at most $c$, for some $c<\infty$ not depending on $n$.

Then from the Cauchy-Schwarz inequality, we get that
\[\mathbb{E}_\lambda \mathbbm{1}(\lambda\in D)\Im m(\lambda+i\eta_u)\le \sqrt{\mathbb{E}_\lambda \left(\Im m(\lambda+i\eta_u)\right)^2 \mathbb{P}_\lambda(\lambda\in D)}\le \sqrt{c\mathbb{P}_\lambda(\lambda\in D)}.\]
Thus, the uniform integrability condition follows.

\medskip
\noindent {\bf Acknowledgments.}  The authors are grateful to Mikl\'os Ab\'ert,  Benjamin Landon, Mikhail Sodin and the anonymous referees for their useful comments. The authors are partially supported by NSERC discovery grant. AM is partially supported by the KKP 139502 project.

\medskip

\noindent \textbf{Data Availability Statement} Data sharing not applicable to this article as no datasets were generated or analysed during the current study.

\appendix

\bibliography{references}

\bibliographystyle{dcu}

\bigskip\bigskip\noindent

\bigskip

\noindent Andr\'as M\'esz\'aros, Department of Computer and Mathematical Sciences, University of Toronto Scarborough, Canada,\\ {\tt a.meszaros@utoronto.ca}

\bigskip

\noindent B\'alint Vir\'ag, Departments of Mathematics and Statistics, University of Toronto, Canada,\\ {\tt balint@math.toronto.edu}

\end{document}